\documentclass{amsart}
\usepackage{a4wide}
\usepackage[latin1]{inputenc}
\usepackage[T1]{fontenc}
\usepackage{amsmath,amssymb,amsthm,stmaryrd}
\usepackage{mathrsfs}
\usepackage{graphicx}
\usepackage{enumerate}
\newcommand{\ud}[0]{\,\mathrm{d}}

\newcommand{\dist}[0]{\operatorname{dist}}
\newcommand{\proj}[0]{\mathbb{P}}
\newcommand{\disc}[0]{\mathbb{D}}
\newcommand{\ball}[0]{\mathbb{B}}

\newcommand{\rGamma}[0]{\reflectbox{\(\Gamma\)}}
\newcommand{\rGammaSmall}[0]{\reflectbox{{\scriptsize\(\Gamma\)}}}

\newcommand{\abs}[1]{|#1|}
\newcommand{\Babs}[1]{\Big|#1\Big|}

\newcommand{\Norm}[2]{\|#1\|_{#2}}
\newcommand{\BNorm}[2]{\Big\|#1\Big\|_{#2}}
\newcommand{\pair}[2]{\langle #1,#2 \rangle}

\newcommand{\ave}[1]{\langle #1\rangle}

\newcommand{\bddlin}[0]{\mathscr{L}}
\newcommand{\kernel}[0]{\mathsf{N}}
\newcommand{\range}[0]{\mathsf{R}}
\newcommand{\domain}[0]{\mathsf{D}}

\newcommand\R{\mathbf{R}}
\newcommand\C{\mathbf{C}}
\newcommand\N{\mathbf{N}}
\newcommand\Z{\mathbf{Z}}

\newcommand{\ran}{\mathsf{R}}

\newcommand{\prob}[0]{\mathbb{P}}
\newcommand{\Exp}[0]{\mathbb{E}}

\newcommand{\radem}[0]{\varepsilon}

\swapnumbers \numberwithin{equation}{section}

\theoremstyle{plain}
\newtheorem{theorem}[equation]{Theorem}
\newtheorem{proposition}[equation]{Proposition}
\newtheorem{corollary}[equation]{Corollary}
\newtheorem{lemma}[equation]{Lemma}

\theoremstyle{definition}
\newtheorem{definition}[equation]{Definition}

\newtheorem{assumption}[equation]{Assumption}

\theoremstyle{remark}
\newtheorem{remark}[equation]{Remark}

\newcommand{\journal}[1]{{\em #1}}
\newcommand{\volume}[1]{{\bfseries #1}}
\newcommand{\name}[1]{{\sc #1}}

\title[Functional calculus of Hodge-Dirac operators]{Holomorphic functional calculus   of   Hodge-Dirac operators in $L^{p}$}

\author[Hyt\"onen]{Tuomas Hyt\"onen}
\address{Department of Mathematics and Statistics, University of Helsinki, Gustaf H\"allstr\"omin katu 2b, FI-00014 Helsinki, Finland}
\email{tuomas.hytonen@helsinki.fi}

\author[McIntosh]{Alan McIntosh} 
\address{  Centre for Mathematics and its Applications,   Australian National University, Canberra ACT 0200, Australia}
\email{Alan.McIntosh@anu.edu.au}

\author[Portal]{Pierre Portal}
\address{Universit\'e Lille 1, Laboratoire Paul Painlev\'e, 59655 Villeneuve d'Ascq, France}
\email{pierre.portal@math.univ-lille1.fr}

\date{\today}


\begin{document}

\begin{abstract}
We study the boundedness of the $H^{\infty}$ functional calculus for differential operators acting in  \(L^{p}(\R^{n};\C^{N})\). For constant coefficients, we give simple conditions on the symbols implying such boundedness. For non-constant coefficients, we extend our recent results for the \(L^p\) theory of the Kato square root problem to the more general framework of Hodge-Dirac operators with variable coefficients \(\Pi_B\) as treated in \(L^2(\R^{n};\C^{N})\) by Axelsson, Keith, and McIntosh.
We obtain a characterization of the property that \(\Pi_B\) has a bounded \(H^{\infty}\) functional calculus, in terms of randomized boundedness conditions of its resolvent. This allows us to deduce stability under small perturbations of this functional calculus.

\end{abstract}

\maketitle

\section{Introduction}
A variety of problems in PDE's can be solved by establishing the boundedness,   and stability under small perturbations,   of the $H^\infty$ functional calculus of certain differential operators.   
In particular, Axelsson, Keith, and McIntosh \cite{akm}   have recovered and extended the solution of the Kato square root problem   \cite{AHLMT}   by showing that Hodge-Dirac operators with variable coefficients of the form
\(\Pi_{B} = \Gamma + B_{1}\Gamma^{*}B_{2}\) have a bounded $H^{\infty}$ functional calculus in $L^{2}(\R^{n};\C^{N})$,
when \(\Gamma\) is a homogeneous first order differential operator with constant coefficients, and 
\(B_{1},B_{2} \in L^{\infty}(\R^{n};\mathscr{L}(\C^{N}))\) are strictly accretive multiplication operators.
Recently,   Auscher, Axelsson, and McIntosh  \cite{aam}   have used related perturbation results to show the openness of some sets of well-posedness for boundary value problems with $L^2$ boundary data. \\

In this paper, we first consider homogeneous differential operators with constant (matrix-valued) coefficients. For such operators the boundedness of  the
$H^{\infty}$ functional calculus is established using Mikhlin's multiplier theorem. However, the estimates on the symbols may be difficult to check in practice, especially when the null spaces of the symbols are non-trivial. Here we provide a simple condition (invertibility of the symbols on their ranges and inclusion of their eigenvalues in a bisector), that gives such estimates.
We then turn to operators with coefficients in $L^{\infty}(\R^{n};\C)$ of the form \(\Pi_{B} = \Gamma + B_{1}\rGamma B_{2}\), where $\Gamma$ and $\rGamma$ are nilpotent homogeneous first order operators with constant (matrix-valued) coefficients, and 
\(B_{1},B_{2} \in L^{\infty}(\R^{n};\mathscr{L}(\C^{N}))\) are multiplication operators satisfying some $L^p$   coercivity   condition. For such operators, we aim at perturbation results which give, in particular, the boundedness of the $H^{\infty}$ functional calculus   when $B_{1},B_{2}$ are small pertubations of constant-coefficient matrices.  \\

This presents two main difficulties. First of all, even in $L^2$, the $H^{\infty}$ functional calculus of a (bi)sectorial operator is   in general   not stable under small perturbations in the sense that there exist a self-adjoint operator $D$ and bounded operators $A$ with arbitrary small norm such that $D(I+A)$ does not have a bounded  $H^{\infty}$ functional calculus (see \cite{my}).
Subtle functional analytic perturbation results exist (see \cite{ddhpv} and \cite{kkw}), but do not give the estimates needed in   \cite{aam} or \cite{akm}.   To obtain such estimates, one needs to take advantage of the specific structure of differential operators using harmonic analytic methods.
Then, the problem of moving from the $L^{2}$ theory to an $L^{p}$ theory is substantial. Indeed, the operators under consideration fall outside the Calder\'on-Zygmund class, and cannot be handled by familiar methods based on interpolation. A known substitute, pioneered by Blunck and Kunstmann in \cite{blunck-kunstmann}, and developed by Auscher and Martell \cite{auscher,auscher-martell1,auscher-martell2,auscher-martell3}, consists in establishing an extrapolation method adapted to the operator, which allows to extend results from $L^2$ to $L^p$ for $p$ in a certain range   $(p_1,p_2)$   containing $2$.
In \cite{hmp} we started another approach, which combines probabilistic tools from functional analysis with the aforementioned $L^2$ methods, and allows $L^p$ results which do not rely on some $L^2$ counterparts.\\

However, our goal in \cite{hmp} was the Kato problem, and we did not reach the generality of \cite{akm} which has recently proven particularly useful in connection with boundary value problems \cite{aam}.
Here we close this gap and, in fact, reach a further level of generality.
Roughly speaking,   for quite general differential operators, we show that   the boundedness of the $H^{\infty}$ functional calculus coincides with the R-(bi)sectoriality (see   Section~\ref{sec:prelim}   for relevant definitions). This then allows perturbation results, in contrast with the general theory of sectorial operators, where R-sectoriality and bounded $H^{\infty}$ calculus are two distinct properties, and perturbation results are much more restricted.\\

For the operators with variable coefficients, the core of the argument is contained in \cite{hmp}, so the reader might want to have a copy of this paper handy.
Here we focus on the points where \cite{hmp} needs to be modified, and develop some adaptation of the techniques to generalized Hodge-Dirac operators which may be of interest in other problems.
To make the paper more readable, we choose not to work in the Banach-space valued setting of \cite{hmp}, but the interested reader will soon realize that our proof carries over to that situation provided   that, as in \cite{hmp}, the target space is a UMD space, and both the space and its dual have the RMF property.   \\

The paper is organized as follows.
In Section 2, we recall the essential definitions. 
In Section 3, we present our setting and state the main results. 
In Section 4, we deal with constant coefficient operators and obtain appropriate estimates on their symbols.
In Section 5, we use these estimates to establish an $L^p$ theory for operators with constant (matrix-valued) coefficients.
In Section 6, we show that a certain (Hodge) decomposition, crucial in our study, is stable under small perturbations.
In Section 7, we give simple proofs of general operator theoretic results on the functional calculus of bisectorial operators.
In Section 8, we   prove   our key results on operators with variable coefficients, referring to \cite{hmp} when arguments are identical, and explaining how to modify them using the results of the preceding sections when they are not. 
  Finally, in Section 9, we derive from Section 8 Lipschitz estimates for the functional calculus of these operators.  

\subsubsection*{Acknowledgments.}
This work advanced through visits of T.H. and P.P. at the Centre for Mathematics and its Applications at the Australian National University, and of P.P. at the University of Helsinki. Thanks go to these institutions for their outstanding support. The research was supported by the Australian Government through the Australian Research Council, and by the Academy of Finland through the project 114374 ``Vector-valued singular integrals''.

\section{Preliminaries}\label{sec:prelim}

Fix some numbers \(n,N \in \Z_+\). We consider functions \(u:\R^n\to\C^N\), or \(A:\R^n\to\bddlin(\C^N)\). The Euclidean norm in both \(\R^n\) and \(\C^N\), as well as the associated operator norm in \(\bddlin(\C^N)\), are denoted by \(\abs{\cdot}\). To express the typical inequalities ``up to a constant'' we use the notation \(a \lesssim b\) to mean that there exists \(C<\infty\) such that \(a \leq Cb\), and the notation \(a \eqsim b\) to mean that \(a \lesssim b \lesssim a\). The implicit constants are meant to be independent of other relevant quantities. If we want to mention that the constant \(C\) depends on a parameter \(p\), we write \(a \lesssim _{p} b\).\\

Let us briefly recall the construction of the $H^{\infty}$ functional calculus (see \cite{adm,cdmy,markusbook,kk,m} for details).
\begin{definition}
A closed operator \(A\) acting in a Banach space \(Y\)  is called \emph{bisectorial} with angle \(\theta\) if its spectrum \(\sigma(A)\) is included in a bisector:
\begin{equation*}\begin{split}
  \sigma(A)\subseteq S_{\theta} &:= \Sigma_{\theta} \cup (-\Sigma_{\theta}),\quad\text{where}\\
 \Sigma_{\theta} &:= \{z \in \C \;;\; |\arg(z)| \leq \theta\},
\end{split}\end{equation*}
and outside the bisector it verifies the following resolvent bounds:
\begin{equation}\label{eq:biSect}
  \forall \theta' \in (\theta,\frac{\pi}{2}) \quad \exists C>0 \quad \forall \lambda \in \C\setminus S_{\theta'} \quad
  \Norm{\lambda(\lambda I-A)^{-1}}{\bddlin(Y)} \leq C.
\end{equation}
\end{definition}
We often omit the angle, and say that \(A\) is   \emph{bisectorial}    if it is bisectorial with \emph{some} angle \(\theta \in [0,\frac{\pi}{2})\). 

For \(0<\nu<\pi/2\), let \(H^{\infty}(S_{\nu})\) be the space of bounded functions on \(S_{\nu}\), which are holomorphic   on the interior of \(S_{\nu}\),   and consider the following subspace of functions with decay at zero and infinity:
\begin{equation*}
  H_{0} ^{\infty}(S_{\nu}) := \Big\{\phi \in H^{\infty}(S_{\nu})\;:\; \exists \alpha,C\in(0,\infty)\quad \forall z \in S_{\nu} \quad \abs{\phi(z)} \leq C\abs{\frac{z}{1+z^{2}}}^{\alpha} \Big\}.
\end{equation*}
For a bisectorial operator \(A\) with angle \(\theta<\omega<\nu<\pi/2\), and \(\psi\in H^{\infty}_0(S_{\nu})\), we define
\begin{equation*}
  \psi(A)u := \frac{1}{2i\pi} \int_{\partial S_{\omega}} \psi(\lambda)(\lambda-A)^{-1}u \ud\lambda,
\end{equation*}  
where \(\partial S_{\omega}\) is  directed anti-clockwise around \(S_{\omega}\).

\begin{definition}
A bisectorial operator \(A\) with angle \(\theta\), is said  to admit a {\em bounded \(H^{\infty}\) functional calculus} with 
angle~\(\mu \in [\theta,\frac{\pi}{2})\) if, for each $\nu \in (\mu,\frac{\pi}{2})$, $$
\quad \exists C<\infty \quad \forall \psi \in H_{0} ^{\infty}(S_{\nu}) \quad \|\psi(A)y\|_{Y} \leq C\|\psi\|_{\infty} \|y\|_{Y}.
$$
\end{definition}
In this case, and if $Y$ is reflexive, one can define a bounded operator $f(A)$ for $f \in H^{\infty}(S_{\nu})$ by
$$
f(A)u := f(0)\mathbb{P}_{0}u + \underset{n \to \infty}{\lim} \psi_{n}(A)u,
$$
where $\mathbb{P}_{0}$ denotes the projection on the null space of $A$   corresponding to the decomposition $Y=\kernel(A)\oplus\overline{\ran(A)}$, which exists for R-bisectorial operators,   and $(\psi_{n})_{n \in \N} \subset H_{0} ^{\infty}(S_{\nu})$ is a bounded sequence which converges locally uniformly to $f$. See \cite{adm,cdmy,markusbook,kk,m} for details. 

\begin{definition}
A family of operators $\mathscr{T}\subset\bddlin(Y)$ is called \emph{R-bounded} if for all \(M \in \N\), all 
\(T_1,\ldots,T_M\in\mathscr{T}\), and all \(u_1,\ldots,u_M \in Y\),
\begin{equation*}
  \Exp  \BNorm{\sum_{k=1} ^{M} \varepsilon_{k}T_k u_{k}}{Y}
  \lesssim \Exp \BNorm{\sum_{k=1} ^{M} \varepsilon_{k}u_{k}}{Y},
\end{equation*}
where   $\Exp$ is   the expectation   which   is taken with respect to a sequence of independent Rademacher variables $\varepsilon_{k}$,   i.e., random signs with $\prob(\varepsilon_k=+1)=\prob(\varepsilon_k=-1)=\frac12$.  

A bisectorial operator $A$ is called {\em R-bisectorial} with angle $\theta$ in $Y$ if the collection
\begin{equation*}
  \{\lambda(\lambda I-A)^{-1}:\lambda\in\C\setminus S_{\theta'}\}
\end{equation*}
is R-bounded for all $\theta'\in(\theta,\pi/2)$.  The infimum of such angles $\theta$ is called the angle of  R-bisectoriality of $A$.
\end{definition}

Again, we may omit the angle and simply say that $A$ is R-bisectorial if it is R-bisectorial with some angle $\theta\in(0,\pi/2)$. Notice that, by a Neumann series argument, this is equivalent to the R-boundedness of $\{(I+itA)^{-1}:t\in\R\}$.
The reader unfamiliar with R-boundedness and the derived notions can consult \cite{hmp} and the references therein.

\begin{remark}
\label{fcRsect}
On subspaces of $L^p$, $1<p<\infty$, an operator with a bounded $H^{\infty}$ functional calculus is R-bisectorial. The proof (stated for sectorial rather than bisectorial operators) can be found in \cite[Theorem 5.3]{kw}.
\end{remark}

\section{Main results}
 
We consider three types of operators.
First, we look at differential operators of arbitrary order with constant (matrix valued) coefficients, and provide simple conditions on their Fourier multiplier symbols to ensure that such operators are bisectorial and, in fact, have a bounded $H^{\infty}$ functional calculus.
Then, we focus on first order operators with a special structure, the Hodge-Dirac operators, and prove that, under an additional condition on the symbols, they give a specific (Hodge) decomposition of $L^p$.
Finally we turn to Hodge-Dirac operators with (bounded measurable) variable coefficients, and show that the boundedness of the $H^{\infty}$ functional calculus is preserved under small perturbation of the coefficients.\\

We work in the Lebesgue spaces \(L^p:= L^p(\R^n;\C^N)\) with \(p\in(1,\infty)\), and denote by $\mathscr{S}(\R^{n};\C^{N})$ the Schwartz class of rapidly decreasing functions with values in $\C^{N}$, and by $\mathscr{S}'(\R^{n};\C^{N})$ the corresponding class of tempered distributions.

\subsection{General constant-coefficient operators}\label{subsec:general}
  
In this subsection, we consider $k$th order homogeneous differential operators of the form
\begin{equation*}
  D=(-i)^k\sum_{\theta\in\N^n:\abs{\theta}=k}\hat{D}_{\theta}\partial^{\theta}
\end{equation*}
acting on $\mathscr{S}'(\R^{n};\C^{N})$ as a Fourier multiplier with symbol 
$\hat{D}(\xi) = \sum_{\abs{\theta}=k}\hat{D}_{\theta}\xi^{\theta}$,
where $\hat{D}_{\theta} \in \mathscr{L}(\C^{N})$.

\begin{assumption}
The Fourier multiplier symbol $\hat{D}(\xi)$ satisfies
\begin{equation*}\tag{D1}\label{D1}
  \kappa\abs{\xi}^{k}\abs{e}\leq\abs{\hat{D}(\xi)e}\quad\text{for all}\quad
  \xi \in \R^{n},\quad\text{all}\quad e\in\range(\hat{D}(\xi)),\quad\text{and some}\quad \kappa>0,
\end{equation*}
\begin{equation*}\tag{D2}\label{D2}
  \text{there exists}\quad \omega \in [0,\frac{\pi}{2})\quad \text{such that for all}\quad \xi \in \R^{n}: \quad
  \sigma(\hat{D}(\xi)) \subseteq S_{\omega}.
\end{equation*}
\end{assumption} 

In each $L^p$,  let $D$ act on its natural domain $\domain_p(D):=\{u \in L^p\;;\; Du \in L^{p}\}$.   In Theorem~\ref{thm:Dagain} we prove: 

\begin{theorem}
\label{thm:D}
Let $1<p<\infty$. Under the assumptions \eqref{D1} and \eqref{D2}, the operator $D$ is bisectorial   in $L^p$ with angle $\omega$,   and has a bounded $H^{\infty}$ functional calculus in $L^p$    with angle $\omega$.   
\end{theorem}

\begin{remark}
 (a)   In \eqref{D2}, the bisector $S_{\omega}$ can be replaced by the sector $\Sigma_{\omega}$   where $0\leq \omega <\pi$.   The operator $D$ is then sectorial
(with   angle $\omega$)   and has a bounded $H^{\infty}$ functional calculus (with   angle $\omega$)   in the sectorial sense, i.e.   $f(D)$ is bounded   for functions $f\in H^{\infty}(\Sigma_{\theta})$ with any $\theta \in (\omega,\pi)$.

 (b)   Assuming that \eqref{D1} holds for all $e \in \C^{N}$ would place us in a more classical context, in which proofs are substantially simpler. We insist on this weaker ellipticity condition since the operators we want to handle have, in general, a non-trivial null space.

 (c)   Using Bourgain's version of Mikhlin's multiplier theorem   \cite{Bourgain:86} , the above theorem extends to functions with values in $X^{N}$, where $X$ is a UMD Banach space.
\end{remark}

\subsection{Hodge-Dirac operators with constant coefficients}\label{ss:unpert}
\label{subsec:unpert}
We now turn to first order   operators   of the form 
$$\Pi = \Gamma + \rGamma,$$
where $
  \Gamma=-i\sum_{j=1}^n\hat{\Gamma}_j\partial_j,$
acts  on $\mathscr{S}'(\R^{n};\C^{N})$ as a Fourier multiplier with symbol 
\begin{equation*}
  \hat{\Gamma}=\hat{\Gamma}(\xi)=\sum_{j=1}^n\hat{\Gamma}_j\xi_j, \quad \hat{\Gamma}_{j} \in \mathscr{L}(\C^{N}),
\end{equation*}
  the operator   $\rGamma$ is defined similarly, and both operators are nilpotent
in the sense that \(\hat{\Gamma}(\xi)^2=0\) and  \(\hat{\rGamma}(\xi)^2=0\) for all \(\xi\in\R^n\).

\begin{definition}
\label{def:unpert}
We call   \(\Pi=\Gamma+\rGamma\)   a \emph{Hodge-Dirac operator} with constant coefficients if its Fourier multiplier symbol \(\hat{\Pi}=\hat\Gamma+\hat\rGamma\) satisfies the following conditions:
\begin{equation*}\tag{$\Pi$1}\label{Pi1}
  \kappa\abs{\xi}\abs{e}\leq\abs{\hat{\Pi}(\xi)e} \quad
  \text{for all}\quad e\in\range(\hat{\Pi}(\xi)),\quad \text{all}\quad \xi \in \R^{n},\quad \text{and some}\quad \kappa>0,
\end{equation*}
\begin{equation*}\tag{$\Pi$2}\label{Pi2}
  \sigma(\hat{\Pi}(\xi)) \subseteq S_{\omega}\quad
  \text{for some}\quad \omega \in [0,\frac{\pi}{2}),\quad\text{and all}\quad \xi \in \R^{n},
\end{equation*}
\begin{equation*}\tag{$\Pi$3}\label{Pi3}
  \kernel(\hat{\Pi}(\xi))=\kernel(\hat{\Gamma}(\xi))\cap \kernel(\hat{\rGamma}(\xi))
  \quad \text{for all}\quad \xi \in \R^{n}.
\end{equation*}
\end{definition}

\begin{remark}
The ``Hodge-Dirac'' terminology has its origins in   applications of this formalism to   Riemannian geometry where $\Gamma$ would be the exterior derivative $d$ and $\rGamma=d^*$. See \cite{akm} for details.  
Note that we are working here in a more general setting than \cite{akm}, where the operator $\rGamma$ was assumed to be the adjoint of $\Gamma$. In particular, our operator $\Pi$ does not need to be self-adjoint   in   $L^2(\R^{n};\C^{N})$.  
\end{remark}

In each $L^p$, we let the operators \(\Upsilon\in\{\Gamma,\rGamma,\Pi\}\) act on their natural domains
\begin{equation*}
  \domain_p(\Upsilon):=\{u\in L^p:\Upsilon u\in L^p\},
\end{equation*}
where \(\Upsilon u\) is defined in the distributional sense. Each \(\Upsilon\) is a densely defined, closed unbounded operator in \(L^p\) with this domain. The formal nilpotence of \(\Gamma\) and \(\rGamma\) transfers into the operator-theoretic nilpotence
\begin{equation*}
  \overline{\range_{p}(\Gamma)}\subseteq\kernel_{p}(\Gamma),\qquad
  \overline{\range_{p}(\rGamma)}\subseteq\kernel_{p}(\rGamma).
\end{equation*}
where $\range_{p}(\Gamma),\kernel_{p}(\Gamma)$ denote the range and kernel of $\Gamma$ as an operator on $L^{p}$.

In Section~\ref{sect:Lptheory} we show that the identity \(\Pi=\Gamma+\rGamma\) is also true in the sense of unbounded operators in \(L^p\).   Moreover, in Theorem~\ref{thm:unpertAgain} we prove: 

\begin{theorem}\label{thm:unpert}
The operator \(\Pi\) has a bounded \(H^{\infty}\) functional calculus in \(L^p\)    with angle $\omega$,  and  satisfies  the Hodge decomposition  
\begin{equation*}
  L^{p} = \kernel_{p}(\Pi)\oplus \overline{\range_{p}(\Gamma)}\oplus \overline{\range_{p}(\rGamma)}.
\end{equation*}
\end{theorem}

\begin{remark}
As in the previous subsection, the above theorem extends to functions with values in $X^{N}$, where $X$ is a UMD Banach space.
\end{remark}

\subsection{Hodge-Dirac operators with variable coefficients}\label{ss:pert}
We finally turn to Hodge-Dirac operators with variable coefficients. The study of such operators is motivated by \cite{aam}, \cite{akm} and \cite{hmp}.

\begin{definition}
\label{def:pert}
Let $1<p<\infty$ and $p'$ denote the dual exponent of $p$. Let
\begin{equation*}
  B_1,B_2\in L^{\infty}(\R^{n};\bddlin(\C^{N})), 
\end{equation*}
and identify these functions with bounded multiplication operators on \(L^p\) in the natural way.   Also let $\Pi=\Gamma+\rGamma$ be a Hodge-Dirac operator. Then
the operator 
\begin{equation}\label{eq:defPiB}
   \Pi_B:=\Gamma+\rGamma_B, \quad  \text{where} \;  \rGamma_B:=B_1\rGamma B_2,
\end{equation}
is called a \emph{Hodge-Dirac operator with variable coefficients} in $L^p$ if the following hold:
\begin{equation*}\tag{B1}
\label{nilpcond}
\rGamma B_{2}B_{1} \rGamma = 0 \quad \text{on} \quad \mathscr{S}(\R^{n};\C^{N}),
\end{equation*}
\begin{equation*}\tag{B2}
\label{coercond}
\Norm{u}{p} \lesssim \Norm{B_{1}u}{p} \quad    \forall u \in \range_{p}(\rGamma) \quad \text{and}  \quad \Norm{v}{p'} \lesssim \Norm{B_2^{*} v}{p'} \quad
\forall v \in \range_{p'}(\rGamma^{*}).
\end{equation*}
\end{definition}

Note that the operator equality \eqref{eq:defPiB}, involving the implicit domain condition $\domain_p(\Pi_B):=\domain_p(\Gamma)\cap\domain_p(\rGamma_B)$, was a proposition for Hodge-Dirac operators with constants coefficients, but is taken as the definition for Hodge-Dirac operators with variable coefficients.

The following simple consequences will be frequently applied. Their proofs are left to the reader. First, the nilpotence condition~\eqref{nilpcond}, a priori formulated for test functions, self-improves to
\begin{equation*}
  \rGamma B_{2}B_{1} \rGamma = 0 \quad \text{on} \quad \domain_p(\rGamma);\quad\text{hence}\quad
  \overline{\range_p(\rGamma_B)}\subseteq\kernel_p(\rGamma_B).
\end{equation*}
Second, the coercivity condition~\eqref{coercond} implies that
\begin{equation*}
  \overline{\range_p(\rGamma_B)}=B_1\overline{\range_p(\rGamma B_2)}=B_1\overline{\range_p(\rGamma)},
\end{equation*}
and
\begin{equation*}
  B_1:\overline{\range_p(\rGamma)}\to\overline{\range_p(\rGamma_B)}\quad
  \text{is an isomorphism}.
\end{equation*} 

Sometimes, we also need to assume that the related operator $\underline{\Pi}_B = \rGamma + B_{2}\Gamma B_{1}$ is a Hodge-Dirac operator with variable coefficients in $L^p$, i.e.
\begin{align*}
\Gamma B_{1}B_{2} \Gamma = 0 \quad &\text{on} \quad \mathscr{S}(\R^{n};\C^{N}),\\
\Norm{u}{p} \lesssim \Norm{B_{2}u}{p} \quad    \forall u \in \range_{p}(\Gamma)\quad  &  \text{and} \quad  
\Norm{v}{p'} \lesssim \Norm{B_1^{*} v}{p'} \quad \forall v \in \range_{p'}(\Gamma^{*}).
\end{align*}

With the same proof as in \cite[Lemma~4.1]{akm}, one can show:

\begin{proposition}
Assuming   $\Pi_B=\Gamma+\rGamma_B$ is a Hodge-Dirac operator with variable coefficients,  
 then the operators \(\rGamma_B:=B_1\rGamma B_2\) and \(\rGamma_B^*:=B_2^*\rGamma^* B_1^*\) are closed, densely defined, nilpotent operators in $L^{p}$ and $L^{p'}$ repectively, and \(\rGamma_B^*=(\rGamma_B)^*\).
\end{proposition}

However, the Hodge-decomposition and resolvent bounds, which in the context of \cite{akm} (and the first-mentioned one even in \cite{hmp}) could be established as propositions, are now properties which may or may not be satisfied:

\begin{definition}
We say that $\Pi_B$ \emph{Hodge-decomposes} $L^p$ if
\begin{equation*}
  L^{p} = \kernel_{p}(\Pi_{B})\oplus \overline{\range_{p}(\Gamma)}\oplus \overline{\range_{p}(\rGamma_{B})}.
\end{equation*}
\end{definition}

\begin{remark}\label{rem:interpolation}
We will mostly be interested in a Hodge-Dirac operator $\Pi_B$ with the property that it is R-bisectorial in $L^p$ and Hodge-decomposes $L^p$. If this property holds for two exponents $p\in\{p_1,p_2\}$, then it holds for the intermediate values $p\in(p_1,p_2)$ as well, and hence the set of exponents $p$, for which the mentioned property is satisfied, is an interval.

The proof that R-bisectoriality interpolates in these spaces can be found in \cite[Corollary 3.9]{kkw}, where it is formulated for R-sectorial operators. As for the Hodge-decomposition, observe first that if a Hodge-Dirac operator $\Pi_B$ is R-bisectorial in $L^p$ and Hodge-decomposes $L^p$, then the projections onto the three Hodge subspaces are given by
\begin{equation*}
  \proj_0=\lim_{t\to\infty}(I+t^2\Pi_B^2)^{-1},\quad
  \proj_{\Gamma}=\lim_{t\to\infty}t^2\Gamma\Pi_B(I+t^2\Pi_B)^{-1},\quad
  \proj_{\rGammaSmall_B}=\lim_{t\to\infty}t^2\rGamma_B\Pi_B(I+t^2\Pi_B)^{-1},
\end{equation*}
where the limits are taken in the strong operator topology. In particular, if $\Pi_B$ has these properties in two different $L^p$ spaces, then the corresponding Hodge subspaces have common projections, and one deduces the boundedness of these projection operators also in the interpolation spaces.
\end{remark}
 
The following main result concerning the operators $\Pi_B$ gives a characterization of the boundedness of their $H^{\infty}$ functional calculus. It will be proven as Corollary~\ref{cor:char} to Theorem~\ref{thm:char}.

\begin{theorem}\label{thm:base}
Let \(1\leq p_{1}<p_{2} \leq \infty\), 
and let \(\Pi_{B}\) be a Hodge-Dirac operator with variable coefficients in $L^p$ which Hodge-decomposes \(L^{p}\) for all \(p \in (p_{1},p_{2})\).
Assume also that $\underline{\Pi}_B$ is a Hodge-Dirac operator with variable coefficients in $L^p$.
Then \(\Pi_{B}\) has a bounded \(H^{\infty}\) functional calculus (with angle $\mu$) in \(L^{p}(\R^{n};\C^{N})\) for all \(p \in (p_{1},p_{2})\) if and only if it is \(R\)-bisectorial (with angle $\mu$) in \(L^{p}(\R^{n};\C^{N})\)  for all \(p \in (p_{1},p_{2})\).
\end{theorem}

This characterization leads to perturbation results such as the following, proven in Corollary~\ref{cor:pert}, thanks to the perturbation properties of R-bisectoriality.

\begin{corollary}\label{cor:base}
Let \(1\leq p_{1}<p_{2} \leq \infty\), and let \(\Pi_{A}\) be a Hodge-Dirac operator with variable coefficients, which is R-bisectorial in $L^p$ and 
Hodge-decomposes \(L^{p}\) for all \(p \in (p_{1},p_{2})\).
Then for each $p\in(p_1,p_2)$, there exists \(\delta=\delta_p>0\) such that,
if \(\Pi_{B}\) and $\underline{\Pi}_B$  are Hodge-Dirac operators with variable coefficients, and 
if   $\Norm{B_{1}-A_{1}}{\infty}+\Norm{B_{2}-A_{2}}{\infty} < \delta$,   then
\(\Pi_{B}\) has a bounded \(H^{\infty}\) functional calculus in \(L^p\) and Hodge-decomposes \(L^{p}\).
\end{corollary}

\begin{remark}
The results in this paper concerning Hodge-Dirac operators with variable coefficients can be extended to the Banach space valued setting, provided the target space   has the so-called UMD and RMF properties, and also its dual has RMF. The UMD property, which passes to the dual automatically, is a well-known notion in the theory of Banach spaces, cf.~\cite{burk}. We introduced the RMF property   in \cite{hmp} in relation with our   \emph{Rademacher maximal function}.   It holds in (commutative or not) $L^p$ spaces for $1<p<\infty$ and in spaces with type 2, and fails in $L^1$. We do not know whether it holds in every UMD space. This paper, especially in Section \ref{sect:proofs}, uses extensively the techniques from \cite{hmp}. We choose not to formulate the results in a Banach space valued setting to make the paper more readable, but all proofs are naturally suited to this more general context. 
\end{remark}

\section{Properties of the symbols}
\label{sec:symbols}
In this section we consider the symbols of the Fourier multipliers defined in Subsections \ref{subsec:general} and \ref{subsec:unpert}.
As a consequence of the assumptions made in these subsections, we obtain the various estimates which we need in the next sections to establish an $L^p$ theory. \\

In what follows, we denote $A(a,b) = \{z \in \C \;;\; a \leq |z| \leq b\}$.

\begin{lemma}
\label{lem:spect}
Let $D$ be a k-th order homogeneous differential operator with constant matrix coefficients, satisfying \eqref{D1} and \eqref{D2}.
Then, denoting $M= \underset{|\xi|=1}{\sup}|\hat{D}(\xi)|$, we have that
\begin{itemize}
\item[(a)]
$
\sigma(\hat{D}(\xi)) \subset (S_{\omega} \cap A(\kappa|\xi|^{k},M|\xi|^{k})) \cup \{0\}, 
$ 
\item[(b)]
$\C^{N}=\kernel(\hat{D}(\xi))\oplus \ran(\hat{D}(\xi))$,
\item[(c)]
$
\forall \mu \in (\omega,\frac{\pi}{2}) \quad |(\zeta I - \hat{D}(\xi))^{-1}| \lesssim |\zeta|^{-1} \quad \forall \xi \in \R^{n}\quad \forall  \zeta  \in
\C\setminus(S_{\mu}\cap A(\tfrac12\kappa|\xi|^{k},2M|\xi|^{k}))$.  
\end{itemize}
\end{lemma}
Using   compactness for $|\xi|=1$ and homogeneity for $|\xi|\neq1$,   this is a consequence of the following lemma.

\begin{lemma}
Let $T \in \mathscr{L}(\C^{N})$,   $\kappa>0$, and $\omega\in[0,\frac\pi2)$, and suppose that   
\begin{itemize}
\item[(i)] \(\kappa\abs{e}\leq\abs{Te}\) for  all \(e\in\range(T)\), and 
\item[(ii)] $
\sigma(T) \subset S_{\omega}.
$
\end{itemize}  
Then  we have that
\begin{itemize}
\item[(a)]
$
\sigma(T) \subset (S_{\omega} \cap A(\kappa, |T|)) \cup \{0\}, 
$  
\item[(b)]
$\C^{N}=\kernel(T)\oplus \ran(T)$,
\item[(c)]$
\forall \mu \in (\omega,\frac{\pi}{2}) \quad |(\zeta I - T)^{-1}| \lesssim |\zeta|^{-1} \quad \forall \zeta  \in \C\setminus(S_{\mu}\cap A(\tfrac12\kappa,2|T|))$.  
\end{itemize}
\end{lemma}

\begin{proof}
Let us first remark that, for a non zero eigenvalue $\lambda$ with eigenvector $e$, we have that
$|\lambda||e| = |Te| \geq \kappa |e|$. This gives (a).
Moreover, (i) also gives that $\kernel(T^{2})=\kernel(T)$. Thus, writing $T$ in Jordan canonical form, we have
the splitting $\C^{N}=\kernel(T)\oplus \ran(T)$. 
The resolvent bounds hold on $\kernel(T)$. On $\ran(T)$, the function   $\zeta \mapsto \zeta(\zeta I - T)^{-1}$ is continuous from the closure of
$\C\setminus(S_{\mu}\cap A(\tfrac12\kappa,2|T|))$ to $\mathscr{L}(\ran(T),\C^{N})$ and is bounded at $\infty$, and thus is bounded on $\C\setminus(S_{\mu}\cap A(\tfrac12\kappa,2|T|))$.  
\end{proof}

Assuming \eqref{D1} and \eqref{D2}, we thus have that   for all $\theta \in (0,\frac{\pi}{2}-\omega)$   and
for all $\xi \in \R^{n}$,
\begin{equation*}
 \exists C>0 \quad \forall \tau \in   S_{\theta}   \quad
  |(I+i\tau\hat{D}(\xi))^{-1}|_{\bddlin(\C^N)} \leq C. 
\end{equation*}
For   $\tau \in S_{\theta}$,   we use the following notation:  
\begin{equation*}\begin{split}
  \hat{R}_\tau(\xi) &:=(I+i\tau\hat{D}(\xi))^{-1},\\
  \hat{P}_\tau(\xi) &:=  \frac12(\hat{R}_\tau(\xi)+\hat{R}_{-\tau}(\xi))= (I+\tau^2\hat{D}(\xi)^2)^{-1},\\
  \hat{Q}_\tau(\xi) &:=  \frac{i}{2}(\hat{R}_{\tau}(\xi)-\hat{R}_{-\tau}(\xi))= \tau\hat{D}(\xi) \hat{P}_\tau(\xi).
\end{split}
\end{equation*}  

  If $\lambda\notin\sigma(\hat{D}(\xi))$ for some $\xi\in\R^n$, then also $\lambda\notin\sigma(\hat{D}(\xi'))$ for all $\xi'$ in some neighbourhood of $\xi$. One checks directly from the definition of the derivative that
\begin{equation*}
  \partial_{\xi_j}(\lambda-\hat{D}(\xi))^{-1}
  =(\lambda-\hat{D}(\xi))^{-1}(\partial_{\xi_j}\hat{D})(\xi)(\lambda-\hat{D}(\xi))^{-1}.
\end{equation*}
By induction it follows that $(\lambda-\hat{D}(\xi))^{-1}$ is actually $C^{\infty}$ in a neighbourhood of $\xi$ for $\lambda\notin\sigma(\hat{D}(\xi))$. In particular, for $\tau\in S_{\theta}$, the function $\hat{R}_\tau(\xi)$ is $C^{\infty}$ in $\xi\in\R^n$, and
\begin{equation}\label{eq:djR}
  \partial_{\xi_j}\hat{R}_\tau(\xi)=\hat{R}_\tau(\xi)(-i\tau\partial_{\xi_j}\hat{D}(\xi))\hat{R}_\tau(\xi).
\end{equation}

\begin{proposition}\label{lem:projEst}
Given the splitting \(\C^N=\kernel(\hat{D}(\xi))\oplus\range(\hat{D}(\xi))\), the complementary projections \(\proj_{\kernel(\hat{D}(\xi))}\) and \(\proj_{\range(\hat{D}(\xi))}=I-\proj_{\kernel(\hat{D}(\xi))}\) are infinitely differentiable in $\R^n\setminus\{0\}$ and satisfy the Mikhlin multiplier conditions
\begin{equation*}
  \abs{\partial_{\xi} ^{\alpha}\proj_{\kernel(\hat{D}(\xi))}}\lesssim_{\alpha}\abs{\xi}^{-\abs{\alpha}},\qquad
  \forall\,\alpha\in\N^n.
\end{equation*}
\end{proposition}

\begin{proof}
The projections \(\proj_{\kernel(\hat{D}(\xi))}\) are obtained by the Dunford--Riesz functional calculus by integrating the resolvent around a contour, which circumscribes the origin and no other point of the spectrum of $\hat{D}(\xi)$. By Lemma~\ref{lem:spect}, for $\xi$ in a neighbourhood of the unit sphere (say $\frac34<\abs{\xi}<\frac43$), we may choose
\begin{equation*}
  \proj_{\kernel(\hat{D}(\xi))}
  =\frac{1}{2\pi i}\int_{\partial\disc(0,2^{-k-1}\kappa)}(\lambda-\hat{D}(\xi))^{-1}\ud\lambda.
\end{equation*}
Using the smoothness of $(\lambda-\hat{D}(\xi))^{-1}$ discussed before the statement of the lemma, differentiation of arbitrary order in $\xi$ under the integral sign may be routinely justified. This shows that $\proj_{\kernel(\hat{D}(\xi))}$ is $C^{\infty}$ in a neighbourhood of the unit sphere.

To complete the proof, it suffices to observe that $\hat{D}(t\xi)=t^k\hat{D}(\xi)$ for $t\in(0,\infty)$. Hence $\kernel(\hat{D}(\xi))$ and $\range(\hat{D}(\xi))$, and therefore the associated projections, are invariant under the scalings $\xi\mapsto t\xi$. It is a general fact that smooth functions, which are homogeneous of order zero, satisfy the Mikhlin multiplier conditions. Indeed, for any $\xi\in\R^n\setminus\{0\}$ and $t\in(0,\infty)$, we have
\begin{equation*}
  \partial_\xi^{\alpha} \proj_{\kernel(\hat{D}(\xi))}
  =\partial_\xi^{\alpha}( \proj_{\kernel(\hat{D}(t\xi))})
  =t^{\abs{\alpha}}(\partial^{\alpha} \proj_{\kernel(\hat{D}(\,\cdot\,))})(t\xi),
\end{equation*}
and setting $t=\abs{\xi}^{-1}$ and using the boundedness of the continuous function $\partial_\xi^{\alpha} \proj_{\kernel(\hat{D}(\xi))}$ on the unit sphere, the Mikhlin estimate follows.
\end{proof}

Notice that, by \eqref{D1}, $\hat{D}(\xi)$ is an isomorphism of $\ran(\hat{D}(\xi))$ onto itself   for each $\xi\in\R^n\setminus\{0\}$.  
We denote by $\hat{D}_{R}^{-1}(\xi)$ its inverse.
 
\begin{lemma}\label{lem:DRest}
The function $\hat{D}_{R} ^{-1}(\xi)\proj_{\range(D(\xi))}$ is smooth in $\R^n\setminus\{0\}$ and satisfies the Mikhlin-type multiplier condition
\begin{equation*}
  \partial_{\xi}^{\alpha}\big(\hat{D}_{R}^{-1}(\xi)\proj_{\range(\hat{D}(\xi))}\big)
  \lesssim_{\alpha} \abs{\xi}^{-k-\abs{\alpha}},\qquad
  \forall\,\alpha\in\N^n.
\end{equation*}
\end{lemma}

\begin{proof}
By the Dunford--Riesz functional calculus, we have
\begin{equation*}
  m(\xi):=\hat{D}_{R} ^{-1}(\xi)\proj_{\range(\hat{D}(\xi))}
  =\frac{1}{2\pi i}\int_{\gamma}(\lambda-\hat{D}(\xi))^{-1}\frac{\ud\lambda}{\lambda},
\end{equation*}
where $\gamma$ is any path oriented counter-clockwise around the non-zero spectrum of $\hat{D}(\xi)$. Since $\hat{D}(r\xi)=r^k\hat{D}(\xi)$, a change of variables and Cauchy's theorem shows that $m(r\xi)=r^{-k}m(\xi)$, and the smoothness of $m$ in $\R^n\setminus\{0\}$ is checked as in the previous proof by differentiating under the integral sign. The modified Mikhlin condition follows from this in a similar way as in the previous proof.
\end{proof}

\begin{proposition}\label{lem:resolvEst}
 Given $\nu \in(0,\frac\pi2-\omega)$, the following Mikhlin conditions hold uniformly in $\tau\in S_{\nu}$:
\begin{equation*}
  \abs{\partial_{\xi} ^{\alpha}\hat{R}_{\tau}(\xi)}\lesssim_{\alpha}\abs{\xi}^{-\abs{\alpha}},\quad
  \forall\,\alpha\in\N^n. 
\end{equation*} 
Similar estimates hold for $\hat{P}_{\tau}(\xi)$ and $\hat{Q}_{\tau}(\xi)$.
\end{proposition}

\begin{proof}
We use induction to establish the desired bounds. For \(\alpha=0\), this was proven in Lemma \ref{lem:spect}. In order to make the induction step, we will need the identity
\begin{equation}\label{eq:dR}
   \partial_{\xi} ^{\alpha}\hat{R}_{\tau}
  =\sum_{0\lneqq\theta\leq\alpha}\binom{\alpha}{\theta}
  (\partial^{\alpha-\theta}\hat{R}_{\tau})(-i\tau\partial_{\xi}^{\theta}\hat{D})\hat{R}_{\tau},
  \qquad\alpha\gneqq 0.
\end{equation}
We use the multi-index notation as follows: the binomial coefficients are $\displaystyle\binom{\alpha}{\theta}:=\prod_{i=1}^n\binom{\alpha_i}{\theta_i}$, the order relation $\theta\leq\alpha$ means that $\theta_i\leq\alpha_i$ for every $i=1,\ldots,n$, whereas $\theta\lneqq\alpha$ means that $\theta\leq\alpha$ but $\theta\neq\alpha$; finally, it is understood that $\displaystyle\binom{\alpha}{\theta}=0$ if $\theta\not\leq\alpha$.

Let us prove the identity \eqref{eq:dR} by induction. For $\abs{\alpha}=1$, the formula was already established in \eqref{eq:djR}.
Assuming \eqref{eq:dR} for some $\alpha\gneqq 0$, we prove it for $\alpha+e_j$, where $e_j$ is the $j$th standard unit vector. Indeed,
\begin{equation*}
\begin{split}
  &\partial_{\xi}^{\alpha+e_j}\hat{R}_{\tau}
  =\partial_{\xi_j}\partial_{\xi} ^{\alpha}\hat{R}_{\tau} \\
  &=\sum_{0\lneqq\theta\leq\alpha}\binom{\alpha}{\theta}\Big[
  (\partial^{\alpha+e_j-\theta}\hat{R}_{\tau})(-i\tau\partial_{\xi}^{\theta}\hat{D})\hat{R}_{\tau}
  +(\partial^{\alpha-\theta}\hat{R}_{\tau})(-i\tau\partial_{\xi}^{\theta+e_j}\hat{D})\hat{R}_{\tau} \\
  &\phantom{=\sum_{0\neq\theta\leq\alpha}\binom{\alpha}{\theta}\Big[}
   +(\partial^{\alpha-\theta}\hat{R}_{\tau})(-i\tau\partial_{\xi}^{\theta}\hat{D})\hat{R}_{\tau}
   (-i\tau\partial_{\xi_j}\hat{D})\hat{R}_{\tau}\Big] \\
  &=\Big[\sum_{0\lneqq\theta\leq\alpha}\binom{\alpha}{\theta}
    +\sum_{e_j\lneqq\theta\leq\alpha+e_j}\binom{\alpha}{\theta-e_j}\Big]
  (\partial^{\alpha+e_j-\theta}\hat{R}_{\tau})(-i\tau\partial_{\xi}^{\theta}\hat{D})\hat{R}_{\tau} 
   +(\partial^{\alpha}\hat{R}_{\tau})(-i\tau\partial_{\xi}^{e_j}\hat{D})\hat{R}_{\tau} \\
  &=\sum_{0\lneqq\theta\leq\alpha+e_j}\Big[\binom{\alpha}{\theta}+\binom{\alpha}{\theta-e_j}\Big]
   (\partial^{\alpha+e_j-\theta}\hat{R}_{\tau})(-i\tau\partial_{\xi}^{\theta}\hat{D})\hat{R}_{\tau},
\end{split}
\end{equation*}
and the proof of \eqref{eq:dR} is completed by the binomial identity $\binom{\alpha}{\theta}+\binom{\alpha}{\theta-e_j}=\binom{\alpha+e_j}{\theta}$. Notice that the induction hypothesis was used twice: first to expand $\partial_{\xi}^{\alpha}\hat{R}_{\tau}$ in the second step, and then to evaluate the summation over the last of the three terms in the third one.

We then pass to the inductive proof of the assertion of the lemma. Let $\alpha\gneqq 0$, and assume the claim proven for all $\beta\lneqq\alpha$. We first consider $(\partial_{\xi} ^{\alpha}\hat{R}_{\tau})\proj_{\range(\hat{D})}$.
By the induction assumption, we know that the factors $\partial^{\alpha-\theta}_{\xi}\hat{R}_{\tau}$ in \eqref{eq:dR} satisfy
\begin{equation*}
  \abs{\partial^{\alpha-\theta}_{\xi}\hat{R}_{\tau}(\xi)}\lesssim\abs{\xi}^{\abs{\theta}-\abs{\alpha}}.
\end{equation*}
Furthermore,
\begin{equation*}
  (-i\tau\partial_{\xi}^{\theta}\hat{D})\hat{R}_{\tau}\proj_{\range(\hat{D})}
  =(\partial_{\xi}^{\theta}\hat{D})\hat{R}_{\tau}(-i\tau\hat{D})\hat{D}_{R}^{-1}\proj_{\range(\hat{D})}
  =(\partial_{\xi}^{\theta}\hat{D})(\hat{R}_{\tau}-I)\hat{D}_{R}^{-1}\proj_{\range(\hat{D})},
\end{equation*}
where the different factors are bounded by
\begin{equation*}
  \abs{\partial_{\xi}^{\theta}\hat{D}}\lesssim\abs{\xi}^{k-\abs{\theta}},\qquad
  \abs{\hat{R}_{\tau}-I}\lesssim 1,\qquad
  \abs{\hat{D}_{R}^{-1}(\xi)\proj_{\range(\hat{D}(\xi))}}\lesssim\abs{\xi}^{-k}.
\end{equation*}
Multiplying all these estimates, we get $\abs{(\partial_{\xi} ^{\alpha}\hat{R}_{\tau})\proj_{\range(\hat{D})}} \lesssim\abs{\xi}^{-\abs{\alpha}}$, as required.

It remains to estimate \((\partial_{\xi} ^{\alpha}\hat{R}_{\tau})\proj_{\kernel(\hat{D})}\). We have
\begin{equation*}\begin{split}
 (\partial_{\xi} ^{\alpha}\hat{R}_{\tau})\proj_{\kernel(\hat{D})}
  &=\partial_{\xi} ^{\alpha}(\hat{R}_{\tau}\proj_{\kernel(\hat{D})})
   -\sum_{0\leq\beta\lneqq\alpha}(\partial_{\xi }^{\beta}\hat{R}_{\tau})\,
     \partial^{\alpha-\beta}\proj_{\kernel(\hat{D})} \\
  &=\partial_{\xi} ^{\alpha}\proj_{\kernel(\hat{D})}
   -\sum_{0\leq\beta\lneqq\alpha}O(\abs{\xi}^{-\abs{\beta}})O(\abs{\xi}^{-\abs{\alpha-\beta}})
  =O(\abs{\xi}^{-\abs{\alpha}}),
\end{split}\end{equation*}
where the \(O\)-bounds follow from the induction assumption and the Mikhlin bounds for $\proj_{\kernel(\hat{D}(\xi))}$ established in Proposition~\ref{lem:projEst}. The proof is complete.
\end{proof}

\begin{proposition}\label{prop:fcPi}  
Let $\mu \in (\omega, \frac{\pi}{2})$,   and let   \(f\in H_0^{\infty}(S_{\mu})\).   Then \(f(D)\) is a Mikhlin multiplier, and more precisely its symbol \(\widehat{f(D)}(\xi)=f(\hat{D}(\xi))\) satisfies, for $\xi \neq 0$,
  \begin{equation*}
  \abs{\partial_{\xi} ^{\alpha}f(\hat{D}(\xi))}\lesssim_{\mu}
  \abs{\xi}^{-\abs{\alpha}}  
  \sup\big\{\abs{f(\lambda)}: \lambda\in S_{\mu},\,\tfrac12\kappa|\xi|^{k}\leq\abs{\lambda}\leq 2M|\xi|^{k}\big\},
\end{equation*}
where $M=\underset{|\xi| =1}{\sup} |\hat{D}(\xi)|$  .
\end{proposition}

\begin{proof}
Pick   $\theta \in (\omega,\mu)$. 
By the definition   of the functional calculus, we have
\begin{equation}\label{eq:fcPi}
  \partial_{\xi} ^{\alpha}\widehat{f({D})}=\frac{1}{2\pi i}\int_{\partial S_{\theta}}f(\zeta)
   \partial_{\xi} ^{\alpha}(\zeta-\hat{D})^{-1}\ud\zeta=\partial_{\xi} ^{\alpha}f(\hat{D}).
\end{equation}
From~\eqref{eq:dR} one sees that   \(\partial_{\xi} ^{\alpha}(\zeta-\hat{D})^{-1}=\zeta^{-1}\partial_{\xi} ^{\alpha}\hat{R}_{i/\zeta}\)   has poles at \(\zeta\in\sigma(\hat{D})\). 
By Lemma \ref{lem:spect}, the non-zero spectrum satisfies
\begin{equation*}
  \sigma(\hat{D}(\xi))\setminus\{0\}\subset  \{\zeta\in S_{\omega}:\kappa|\xi|^{k}\leq\abs{\zeta}\leq M|\xi|^{k}\}. 
\end{equation*}
Hence, for a fixed \(\xi\in\R^n\), we may deform the integration path in~\eqref{eq:fcPi} to  
\begin{equation*}
  \partial [S_{\theta}\cap \disc(0,2M|\xi|^{k})\setminus\disc(0,\tfrac12\kappa|\xi|^{k})]\cup
  \partial [S_{\theta}\cap \disc(0,\varepsilon)]=:\gamma_1\cup\gamma_2.
\end{equation*}  
On \(\gamma_1\), there holds \(\abs{\zeta}\eqsim|\xi|^{k}\), while the length of the path is also \(\ell(\gamma_1)\eqsim|\xi|^{k}\). Hence  
\begin{equation*}\begin{split}
  &\Babs{\int_{\gamma_1}f(\zeta)
   \partial_{\xi} ^{\alpha}(\zeta-\hat{D}(\xi))^{-1}\ud\zeta} \\
 &\leq\sup\{\abs{f(\zeta)}:\zeta\in S_{\mu}\cap\disc(0,2M|\xi|^{k})\setminus\disc(0,\tfrac12\kappa|\xi|^{k})\}
     \int_{\gamma_1}\abs{\partial_{\xi} ^{\alpha}\hat{R}_{i/\zeta}(\xi)}\frac{\abs{\ud\zeta}}{\abs{\zeta}} \\
  &\lesssim  \sup\big\{\abs{f(\lambda)}: \lambda\in S_{\mu},\,\tfrac12\kappa|\xi|^{k}\leq\abs{\lambda}\leq 2M|\xi|^{k}\big\}
\cdot\abs{\xi}^{-\abs{\alpha}},
\end{split}\end{equation*} 
which has a bound of the desired form. On the other hand, the integral over \(\gamma_2\) vanishes in the limit as \(\varepsilon\downarrow 0\).
\end{proof}

We shall make use of operators with the following particular symbols:

\begin{corollary}\label{cor:sigmaLambda}
For $t \in \R$ and $\theta\in\N^k$ with $\abs{\theta}=k$, the symbols
\begin{equation*}
  \sigma_{t}(\xi)
  :=t^2\xi^{\theta}\hat{D}(\xi)(I+t^{2}\hat{D}(\xi)^2)^{-1}
\end{equation*}  
are \(C^{\infty}\) away from the origin and satisfy the Mikhlin multiplier estimates
\begin{equation*}
  \abs{\partial_{\xi}^{\alpha}\sigma_t(\xi)}
  \lesssim_{\alpha}\abs{\xi}^{-\abs{\alpha}}\qquad\forall\ \alpha\in\N^n.
\end{equation*}
\end{corollary}

\begin{proof}
The symbols are
\begin{equation*}
  \sigma_{t}(\xi)=\xi^{\theta}\hat{D}_R^{-1}(\xi)t^2\hat{D}^2(\xi)(I+t^{2}\hat{D}(\xi)^2)^{-1}
  =\big(\xi^{\theta}\hat{D}_R^{-1}(\xi)\proj_{\range(\hat{D}(\xi))}\big)
   \big(I-\hat{P}_t(\xi)\big),
\end{equation*}
where both factors are smooth in $\R^n\setminus\{0\}$, and the first satisfies the Mikhlin multiplier condition by Lemma~\ref{lem:DRest} and the second by Proposition~\ref{lem:resolvEst}.
\end{proof}

\section{$L^p$ theory for operators with constant coefficients}\label{sect:Lptheory}

In this section, we consider the Fourier multipliers in $L^p$, which correspond to the symbols studied in Section~\ref{sec:symbols}.
We denote by $R_{t},P_{t},Q_{t}$ the multiplier operators with the symbols $\hat{R}_{t},\hat{P}_{t},\hat{Q}_{t}$.
We start with the operators $D$ from Subsection \ref{subsec:general}. The estimates obtained in the preceding section give Theorem~\ref{thm:D}, restated here for convenience:  

\begin{theorem}\label{thm:Dagain}
Let $1<p<\infty$. Under the assumptions \eqref{D1} and \eqref{D2}, the operator $D$ is bisectorial   in $L^p$ with angle $\omega$,   and has a bounded $H^{\infty}$ functional calculus in $L^p$    with angle $\omega$.   
\end{theorem}

\begin{proof}
By   the   Mikhlin multiplier theorem, the bisectoriality follows from Proposition \ref{lem:resolvEst}, while the boundedness of the $H^\infty$ functional calculus follows from Proposition \ref{prop:fcPi}.
\end{proof}
 
The coercivity condition \eqref{D1} for the symbol has the following reincarnation on the level of operators:

\begin{proposition}\label{prop:coer}
For all \(u\in\overline{\range_{p}(D)}\cap\domain_p(D)\), there holds \(u\in\domain_p(\nabla^k)\), and
\begin{equation*}
  \Norm{\nabla^k u}{p}\lesssim\Norm{D u}{p}.
\end{equation*}
\end{proposition}

\begin{proof}
For \(u\in\domain_p(D)\cap\overline{\range_{p}(D)}\), we have (for real $t$) 
\begin{equation*}
  u_{t}:=t^2 DP_t (D u)=t^2D^2 P_t u
  =(I-P_t) u \to \proj_{\overline{\range_{p}(D)}}u=u,\qquad t\to\infty.
\end{equation*}
The operators \(t^2\partial^{\theta} DP_t\), \(\abs{\theta}=k\), are bounded on \(L^p\) by Corollary~\ref{cor:sigmaLambda}. It follows that \(u_{t}\in\domain_p(\partial^{\theta})\), and 
\begin{equation*}
  \Norm{\partial^{\theta} u_{t}}{p}
  =\Norm{t^2\partial^{\theta} DP_t(D u)}{p}
  \lesssim\Norm{Du}{p}.
\end{equation*}

Let \(w\) be a test function in the dual space. Then
\begin{equation*}\begin{split}
  \abs{\pair{\partial^{\theta} u}{w}}
  =\abs{\pair{u}{\partial^{\theta} w}}
  &=\lim_{t\to\infty}\abs{\pair{u_t}{\partial^{\theta} w}}
  =\lim_{t\to\infty}\abs{\pair{\partial^{\theta} u_{t}}{w}} \\
  &\leq\liminf_{t\to\infty}\Norm{\partial^{\theta} u_{t}}{p}\Norm{w}{p'}
  \lesssim\Norm{D u}{p}\Norm{w}{p'}.
\end{split}\end{equation*}
Thus \(u\in\domain_p(\partial^{\theta})\) and \(\Norm{\partial^{\theta} u}{p}\lesssim\Norm{D u}{p}\) for all \(\abs{\theta}=k\).
\end{proof}

We turn to the Hodge-Dirac operators $\Pi=\Gamma+\rGamma$ which satisfy the hypotheses at the start of Subsection \ref{subsec:unpert}, and note that they then satisfy the conditions on $D$ with $k=1$. In particular, Theorem \ref{thm:Dagain} and Proposition \ref{prop:coer} hold for $D=\Pi$ and $k=1$. Moreover there is an operator version of the symbol condition \eqref{Pi3}: 

\begin{lemma}\label{lem:kernelInters}
There holds \(\kernel_{p}(\Pi)=\kernel_{p}(\Gamma)\cap\kernel_{p}(\rGamma)\).
\end{lemma}

\begin{proof}
The inclusion \(\supseteq\) is clear. Let \(u\in\kernel_{p}(\Pi)\). Then \(\hat\Pi\hat{u}=0\) in the sense of distributions. It follows from Lemma~\ref{lem:DRest} that the function \(\Upsilon(\xi):=\hat{\Gamma}(\xi)\hat{\Pi}_{R}^{-1}(\xi)\proj_{\range(\hat{\Pi}(\xi))}\) is \(C^{\infty}\) away from the origin. Hence, if \(\psi\in C_c^{\infty}(\R^n\setminus\{0\})\), then also \(\psi\Upsilon\) is in the same class, and the product \(\psi\Upsilon\cdot\hat\Pi\hat{u}\) is well-defined and vanishes as a distribution. But
\begin{equation*}
  \Upsilon(\xi)\hat\Pi(\xi)
  =\hat{\Gamma}(\xi)\hat{\Pi}_{R}^{-1}(\xi)\proj_{\range(\hat{\Pi}(\xi))}\hat\Pi(\xi)
  =\hat{\Gamma}(\xi)\proj_{\range(\hat{\Pi}(\xi))}
  =\hat{\Gamma}(\xi)
\end{equation*}
by \eqref{Pi3}. Thus we have shown that \(\psi\hat{\Gamma}\hat{u}=\psi\Upsilon\hat{\Pi}\hat{u}=0\) for every \(\psi\in C_c^{\infty}(\R^n\setminus\{0\})\). This means that the distribution \(\hat\Gamma\hat{u}\) is at most supported at the origin, and hence \(\Gamma u=P\), a polynomial. But also \(\Gamma u=-i\sum_{j=1}^n\hat\Gamma_j\partial_j u=\sum_{j=1}^n\partial_j u_j\), where \(u_j=-i\hat\Gamma_j u\in L^p\). Let \(\hat\phi\in\mathscr{S}(\R^n)\) be identically one in a neighbourhood of the origin. Then \(\hat{P}=\hat{P}\hat\phi\) and hence \(P=P*\phi\). But
\begin{equation*}
  P*\phi(y)=\pair{P}{\phi(y-\cdot)}
  =\sum_{j=1}^n\pair{u_j}{(\partial_j\phi)(y-\cdot)}
  \to 0
\end{equation*}
as \(\abs{y}\to\infty\) (using just the fact that \(u_j\in L^p\) and \(\partial_j\phi\in L^{p'}\)), and a polynomial with this property must vanish identically. This shows that \(\Gamma u=P=0\), and then also \(\rGamma u=\Pi u-\Gamma u=0\).
\end{proof}

This then implies: 

\begin{proposition}
The operator identity $\Pi = \Gamma +\rGamma$ holds in $L^p$, 
in the sense that \(\domain_p(\Pi)=\domain_p(\Gamma)\cap\domain_p(\rGamma)\) and \(\Pi u=\Gamma u+\rGamma u\) for all \(u\in\domain_p(\Pi)\).
\end{proposition}

\begin{proof}
It is clear that \(\domain_p(\Gamma)\cap\domain_p(\rGamma)\subseteq\domain_p(\Pi)\).
Since \(\Pi\) is bisectorial in \(L^p\), there is the topological decomposition \(L^p=\kernel_{p}(\Pi)\oplus\overline{\range_{p}(\Pi)}\). Write \(u\in\domain_p(\Pi)\) as \(u=u_0+u_1\) in this decomposition. Then \(u_0\in\kernel_{p}(\Pi)=\kernel_{p}(\Gamma)\cap\kernel_{p}(\rGamma)\), and \(u_1=u-u_0\in\domain_p(\Pi)\cap\overline{\range_{p}(\Pi)}\). By Proposition~\ref{prop:coer}, \(u_1\in\domain_p(\nabla)\subseteq\domain_p(\Gamma)\cap\domain_p(\rGamma)\). Thus also \(\domain_p(\Pi)\subseteq\domain_p(\Gamma)\cap\domain_p(\rGamma)\). The coincidende of \(\Pi u\) and \(\Gamma u+\rGamma u\) is clear from the distributional definition.
\end{proof}

We are ready to prove Theorem~\ref{thm:unpert}, restated here: 

\begin{theorem}\label{thm:unpertAgain} Let $\Pi$ be a Hodge-Dirac operator with constant coefficients, and let $1<p<\infty$. Then
the operator \(\Pi\) has a bounded \(H^{\infty}\) functional calculus in \(L^p\)   with angle $\omega$,   and    satisfies   the following Hodge decomposition
\begin{equation*}
  L^{p} = \kernel_{p}(\Pi)\oplus \overline{\range_{p}(\Gamma)}\oplus \overline{\range_{p}(\rGamma)},
\end{equation*} 
where $\overline{\range_{p}(\Pi)}=\overline{\range_{p}(\Gamma)}\oplus\overline{\range_{p}(\rGamma)}$. 
\end{theorem}

\begin{proof}
The fact that $\Pi$ is a bisectorial operator with a bounded $H^{\infty}$ functional calculus is a particular case of Theorem \ref{thm:D}.   The bisectoriality already implies the decomposition $L^p=\kernel_p(\Pi)\oplus\overline{\range_p(\Pi)}$, which we now want to refine.

We first check that $\overline{\range_p(\Pi)}\subseteq\overline{\range_p(\Gamma)}+\overline{\range_p(\rGamma)}$. If \(u\in\overline{\range_{p}(\Pi)}\), then \(u=\lim_{j\to\infty}\Pi y_j\) for some \(y_j\in\domain_{p}(\Pi)\cap\overline{\range_{p}(\Pi)}\subseteq\domain_{p}(\nabla)\). Then $\Norm{\Gamma(y_j-y_k)}{p}\lesssim\Norm{\nabla(y_j-y_k)}{p}\lesssim\Norm{\Pi(y_j-y_k)}{p}\to 0$ (using Proposition~\ref{prop:coer}), and hence \(\Gamma y_j\) converges to some \(v\in\overline{\range_{p}(\Gamma)}\) with \(\Norm{v}{p}\lesssim\Norm{u}{p}\). Similarly, \(\rGamma y_j\) converges to \(w\in\overline{\range_{p}(\rGamma)}\), and \(u=v+w\in\overline{\range_{p}(\Gamma)}+\overline{\range_{p}(\rGamma)}\).

Next we show that $\range_p(\Gamma)\subseteq\range_p(\Pi)$. Indeed,
$\range_p(\Gamma) = \Gamma(\domain_p(\Gamma)) = \Gamma (\domain_p(\Gamma)\cap \overline{\range_p(\Pi)})$ (by the decomposition in the first paragraph and Lemma~\ref{lem:kernelInters}) $\subseteq\Gamma(\domain_p(\Gamma)\cap(\overline{\range_p(\Gamma)}+\overline{\range_p(\rGamma)}))$ (by the previous paragraph) $=\Gamma(\domain_p(\Gamma)\cap\overline{\range_p(\rGamma)})$ (because $\Gamma$ is nilpotent) $=\Pi(\domain_p(\Pi)\cap\overline{\range_p(\rGamma)})$ (because $\rGamma$ is nilpotent) $\subseteq \range_p(\Pi)$.  Therefore $\overline{\range_p(\Gamma)}\subseteq\overline{\range_p(\Pi)}$. In a similar way, we see that $\overline{\range_p(\rGamma)}\subseteq\overline{\range_p(\Pi)}$.

On combining these two results with that in the preceding paragraph, we obtain
\(\overline{\range_{p}(\Gamma)}+\overline{\range_{p}(\rGamma)}=\overline{\range_{p}(\Pi)}\).
To show that the sum  is direct, observe that
\begin{equation*}
  \overline{\range_{p}(\Gamma)}\cap\overline{\range_{p}(\rGamma)}
  \subseteq\overline{\range_{p}(\Pi)}\cap\kernel_{p}(\Gamma)\cap\kernel_{p}(\rGamma)
  =\overline{\range_{p}(\Pi)}\cap\kernel_{p}(\Pi)
  =\{0\}
\end{equation*}
by nilpotence, Lemma~\ref{lem:kernelInters} and the decomposition \(L^p=\kernel_{p}(\Pi)\oplus\overline{\range_{p}(\Pi)}\).
\end{proof} 

We conclude this section with an analogue, in our matrix-valued context, of Bourgain's  \cite[Lemma~10]{Bourgain:86}.   
It is an important property of Hodge-Dirac operators with constant coefficients   which   we use to study Hodge-Dirac operators with variable coefficients  in Section \ref{sect:proofs}. 

\begin{proposition}\label{prop:transQuadEst}
For all \(z\in\R^n\), there holds
\begin{equation*}
  \Exp\BNorm{\sum_k\radem_k \tau_{2^k z} Q_{2^k} u}{p}
  \lesssim(1+\log_+\abs{z})\Norm{u}{p}.
\end{equation*} 
  where $\tau_{z}$ denotes the translation operator $\tau_{z}u(x):=  u(x-z) $. 
\end{proposition}
 
\begin{proof}
Let us fix a test function \(\varphi\in\mathscr{D}(\R^n)\) such that \(1_{\ball(0,2^{-1})}\leq\varphi\leq 1_{\ball(0,1)}\), and write \(\psi(\xi):=\varphi(\xi)-\varphi(2\xi)\), \(\psi_m(\xi):=\psi(2^{-m}\xi)\), and \(\varphi_m(\xi):=\varphi(2^{-m}\xi)\) for \(m\in\Z_+\). Let \(\Phi_m\) and \(\Psi_m\), \(m\in\Z\), be the corresponding Fourier multiplier operators with symbols \(\varphi_m\) and \(\psi_m\). Then we have the partition of unity \(\varphi_k+\sum_{m=1}^{\infty}\psi_{k+m}\equiv 1\), and the corresponding operator identity \(\Phi_k+\sum_{m=1}^{\infty}\Psi_{m+k}=I\) for all \(k\in\Z\).

Since the support of the Fourier transform of \(Q_{2^k}\Psi_{m-k}\) is contained in \(\ball(0,2^{m-k})\), by Bourgain's   \cite[Lemma~10]{Bourgain:86},   there holds
\begin{equation}\label{eq:usedBourgain}
  \Exp\BNorm{\sum_k\radem_k \tau_{2^k z} Q_{2^k} \Psi_{m-k} u}{p}
  \lesssim(1+\log_+(2^m\abs{z}))\Exp\BNorm{\sum_k\radem_k Q_{2^k} \Psi_{m-k} u}{p}.
\end{equation}
The same reasoning applies with \(\Phi\) in place of \(\Psi\).

We now estimate the right side of \eqref{eq:usedBourgain} as a Fourier multiplier transformation. The symbol of the operator acting on \(u\) is given by
\begin{equation*}
  \sigma(\xi)=\sum_k\radem_k f(2^k \hat{\Pi}(\xi))\psi(2^{k-m}\xi),\qquad
  f(\tau)=\tau(1+\tau^2)^{-1}.
\end{equation*}
For every \(\alpha\in\{0,1\}^n\), a computation shows that
\begin{equation*}
  \partial^{\alpha}\sigma(\xi)=\sum_k\radem_k \sum_{\beta\leq\alpha}
   \partial^{\beta} f(2^k \hat{\Pi}(\xi))(\partial^{\alpha-\beta}\psi)(2^{k-m}\xi)2^{(k-m)\abs{\alpha-\beta}}.
\end{equation*}
By the support property of \(\psi\), the series in \(k\) reduces to at most two non-vanishing terms for which \(2^{k-m}\abs{\xi}\eqsim 1\). By Proposition~\ref{prop:fcPi}, there moreover holds
\begin{equation}\label{eq:auxBdFc}
  \abs{\partial^{\beta} f(2^k \hat{\Pi}(\xi))}
  \lesssim\frac{2^k\abs{\xi}}{1+(2^k\abs{\xi})^2}\abs{\xi}^{-\abs{\beta}},
\end{equation}
which shows that, for \(m\in\Z_+\),
\begin{equation*}
  \abs{\partial^{\alpha}\sigma(\xi)}
  \lesssim 2^{-m}\abs{\xi}^{-\abs{\alpha}}.
\end{equation*}
Hence the associated Fourier multiplier is bounded with norm \(\lesssim 2^{-m}\).

A similar computation can be made with \(\Phi_{-k}\) in place of \(\Psi_{m-k}\), but the estimation of the symbol then involves an infinite series of terms:
\begin{equation*}\begin{split}
  \abs{\partial^{\alpha}\sigma(\xi)}
  &\lesssim\sum_{\beta\leq\alpha}\sum_k\abs{\partial^{\beta} f(2^k \hat{\Pi}(\xi))}
   \abs{(\partial^{\alpha-\beta}\varphi)(2^k\xi)}2^{k\abs{\alpha-\beta}} \\
  &\lesssim\sum_{\beta\leq\alpha}\sum_{k:|2^k\xi|\leq 1}2^{k(1+\abs{\alpha-\beta})}\abs{\xi}^{1-\abs{\beta}}
   \lesssim\abs{\xi}^{-\abs{\alpha}},
\end{split}\end{equation*}
where \eqref{eq:auxBdFc} was used again in the second estimate. We conclude that the operator acting on \(u\) in \eqref{eq:usedBourgain}, with \(\Phi_{-k}\) in place of \(\Psi_{m-k}\), is also bounded.

Collecting all the estimates, we have shown that
\begin{equation*}\begin{split}
  \Exp &\BNorm{\sum_k\radem_k \tau_{2^k z} Q_{2^k} u}{p} \\
  &\lesssim(1+\log_+\abs{z})\BNorm{\sum_k\radem_k Q_{2^k}\Phi_{-k} u}{p}
   +\sum_{m=1}^{\infty}(m+\log_+\abs{z})\BNorm{\sum_k\radem_k Q_{2^k}\Psi_{m-k} u}{p} \\
  &\lesssim\sum_{m=0}^{\infty}(\max\{1,m\}+\log_+\abs{z})2^{-m}\Norm{u}{p}
   \lesssim(1+\log_+\abs{z})\Norm{u}{p},
\end{split}\end{equation*}
which is the asserted bound.
\end{proof}

\section{Properties of Hodge decompositions}\label{sec:hodge}

In this section, we collect various results concerning the Hodge decomposition. These include duality results, a relation of the Hodge decompositions of the operator $\Pi_B$ and its variant $\underline{\Pi}_B$, and finally some stability properties of the Hodge decomposition under small perturbations of the coefficient matrices $B_1$ and $B_2$. These will be needed in proving the stability of the functional calculus of the Hodge-Dirac operators under small perturbations later on.

\begin{lemma}
\label{lem:Hodge2}
Let \(1<p<\infty\) and let \(\Pi_{B}\) be a Hodge-Dirac operator with variable coefficients in $L^p$. The following assertions are equivalent:
\begin{enumerate}
\item\label{it:Hodge} \(\Pi_{B}\) Hodge-decomposes \(L^{p}\).
\item\label{it:Hodge2}
\(
\begin{cases}
  L^p = \kernel_{p}(\Gamma) \oplus \overline{\range_{p}(\rGamma_B)},\\
  L^p = \kernel_{p}(\rGamma _B) \oplus \overline{\range_{p}(\Gamma)}.
\end{cases}
\)
\end{enumerate}
\end{lemma}

\begin{proof}
  \(\eqref{it:Hodge}\Rightarrow\eqref{it:Hodge2}\).  
We first show that $\kernel_p(\Pi_{B}) = \kernel_p(\Gamma) \cap \kernel_p(\rGamma_{B})$.
If $u \in \kernel_p(\Pi_{B})$
then $\Gamma u = - \rGamma_{B}u \in \overline{\range_p(\Gamma)} \cap \overline{\range_p(\rGamma_{B})} = \{0\}$.
It follows that $\kernel_p(\Pi_{B}) \oplus \overline{\range_p(\Gamma)} \subseteq \kernel_p(\Gamma)$.
Also $\kernel_p(\Gamma) \cap \overline{\range_p(\rGamma_{B})}
 \subseteq \kernel_p(\Pi_{B})\cap  \overline{\range_p(\rGamma_{B})}  = \{0\}$.
Hence $\kernel_p(\Pi_{B}) \oplus \overline{\range_p(\Gamma)} = \kernel_p(\Gamma)$,
and thus $L^{p} = \kernel_p(\Gamma)\oplus \overline{\range_p(\rGamma_{B})}$.
  The second part of \eqref{it:Hodge2} is similarly proven.  

\(\eqref{it:Hodge2}\Rightarrow\eqref{it:Hodge}\). By \eqref{it:Hodge2}, $u\in L^p$ can be decomposed as \(u = v_0 + v_1 + u_1\) where \(v_0+v_1 \in \kernel_{p}(\Gamma)\), \(v_0 \in \kernel_{p}(\rGamma_B)\), 
\(v_1 \in \overline{\range_{p}(\Gamma)}\), and  \(u_1 \in \overline{\range_{p}(\rGamma_B)}\).
Then \(\Pi_B v_0 = \Gamma (v_0+v_1-v_1) = 0\),  and \(\Norm{v_0}{p}+\Norm{v_1}{p} \lesssim \Norm{v_0+v_1}{p} \lesssim \Norm{u}{p}\).
Moreover
\begin{equation*}\begin{split}
  &\kernel_{p}(\Pi_B) \cap \overline{\range_{p}(\Gamma)} \subseteq \kernel_{p}(\rGamma_B) \cap \overline{\range_{p}(\Gamma)}=\{0\},\\
  &\kernel_{p}(\Pi_B) \cap \overline{\range_{p}(\rGamma_B)} \subseteq \kernel_{p}(\Gamma) \cap \overline{\range_{p}(\rGamma_B)}=\{0\},\\
  &\overline{\range_{p}(\rGamma_B)}\cap \overline{\range_{p}(\Gamma)} \subseteq \kernel_{p}(\Gamma) \cap \overline{\range_{p}(\rGamma_B)}=\{0\}.
\end{split}\end{equation*}
The proof is complete.
\end{proof}

\begin{lemma}\label{lem:HodgeDuality}
Let \(D_0\) and \(D_1\) be closed, densely defined operators in \(L^p\). Then
\begin{equation*}
  L^p=\kernel_{p}(D_0)\oplus\overline{\range_{p}(D_1)}\qquad\text{if and only if}\qquad
  L^{p'}=\kernel_{p'}(D_1^*)\oplus\overline{\range_{p'}(D_0^*)}.
\end{equation*}
\end{lemma}

\begin{proof}
Assuming the first decomposition, we have that
$$
 L^{p'} = \overline{\range_{p}(D_{1})}^{\perp} \oplus \kernel_{p}(D_0)^{\perp} = \kernel_{p'}(D_1^{*}) \oplus  \overline{\range_{p'}(D_{0}^{*})}.
$$
The other implication follows by symmetry.
\end{proof}

\begin{lemma}\label{lem:dual}
Let $\Pi_{B}$ be a Hodge-Dirac operator with variable coefficients in $L^p$ which Hodge decomposes $L^p$. Then $\Pi_{B}^{*}$ is a Hodge-Dirac operator with variable coefficients in $L^{p'}$ which Hodge decomposes $L^{p'}$.
\end{lemma}
\begin{proof}
We have to show that $B^{*}=(B_{1}^{*},B_{2}^{*})$ satisfy (\ref{coercond}) in $L^{p'}$. 
Let us remark that $B_1$ is an isomorphism from $\overline{\range_p(\rGamma)}$ onto $\overline{\range_p(\rGamma_{B})}$. By Lemma \ref{lem:Hodge2}, this means that $B_1$ is an isomorphism from $\overline{\range_p(\rGamma)}$ onto $L^{p}/\kernel_p(\Gamma)$, and thus 
$B_1^{*}$ is an isomorphism from $\range_{p'}(\Gamma^{*})$ onto $L^{p'}/\kernel_{p'}(\rGamma^{*})$. This gives the part of the result concerning $B_{1}^{*}$ 
thanks to Lemma \ref{lem:HodgeDuality}. The case of $B_{2}^{*}$ is handled in the same way. Condition (\ref{nilpcond}) is obtained by duality, and the proof is concluded by applying Lemma 
\ref{lem:Hodge2} and Lemma \ref{lem:HodgeDuality}.
\end{proof}

Recall that $\underline{\Pi}_B:=\rGamma+B_2\Gamma B_1$.

\begin{lemma}\label{lem:tilde}  
Let \(1<p<\infty\) and suppose   that   \(\Pi_{B}\) and \(\underline{\Pi}_{B}\) are both Hodge-Dirac operator with variable coefficients in $L^p$, and that \(\Pi_{B}\) Hodge-decomposes \(L^{p}\). Then  \(\underline{\Pi}_{B}\) also Hodge-decomposes \(L^{p}\). If, moreover,   \(\Pi_{B}\) is an R-bisectorial operator   in   $L^p$, then so is \(\underline{\Pi}_{B}\).   
\end{lemma}

\begin{proof}   
By Lemmas~\ref{lem:Hodge2} and \ref{lem:HodgeDuality}, the assumption   that \(\Pi_{B}\) Hodge-decomposes 
\(L^{p}\)   is equivalent to
\begin{equation}\label{eq:HodgeAss}
  L^p=\kernel_{p}(\Gamma) \oplus \overline{\range_{p}(\rGamma_B)},\qquad
  L^{p'}=\kernel_{p'}(\Gamma^*) \oplus \overline{\range_{p'}(\rGamma_B^*)},
\end{equation}
whereas the claim   that \(\underline{\Pi}_{B}\) Hodge-decomposes 
\(L^{p}\)   is equivalent to
\begin{equation}\label{eq:HodgeClaim}
  L^p=\kernel_{p}(B_2\Gamma B_1) \oplus \overline{\range_{p}(\rGamma)},\qquad
  L^{p'}=\kernel_{p'}(B_1^*\Gamma^* B_2^*) \oplus \overline{\range_{p'}(\rGamma^*)}.
\end{equation}

We show that the first decomposition in \eqref{eq:HodgeAss} implies the first one in \eqref{eq:HodgeClaim}.
Let \(u \in L^p\) and write \(B_1u = v + w\) where \(v \in \kernel_{p}(\Gamma)\) and \(w \in \overline{\range_{p}(\rGamma_B)}\).
Let \(w = B_1 x\) for \(x \in  \overline{\range_{p}(\rGamma)}\). Then \(u-x \in \kernel_{p}(B_2\Gamma B_1)\) and
\begin{equation*}
  \Norm{x}{p} \lesssim \Norm{B_1x}{p} = \Norm{w}{p} \lesssim \Norm{B_1 u}{p} \lesssim \Norm{u}{p}.
\end{equation*}
The deduction of the second decomposition in \eqref{eq:HodgeClaim} from the second one in \eqref{eq:HodgeAss} is analogous.

We now turn to the second statement. Let us denote by $\underline{\mathbb{P}}_{1}$ the projection on $\overline{\range_p(\rGamma)}$, by $\underline{\mathbb{P}}_{2}$ the projection on $\overline{\range_p(B_{2}\Gamma B_{1})}$, by $\mathbb{P}_{1}$ the projection on $\overline{\range_p(\Gamma)}$, and by $\mathbb{P}_{2}$ the projection on $\overline{\range_p(B_{1}\rGamma B_{2})}$. 
Let $(t_{k})_{k\in \N} \subset \R$ and $(u_{k})_{k \in \N} \subset L^{p}$, and remark first that
\begin{equation*}
  (I+it_{k}\underline{\Pi}_{B})^{-1} = I - (t_{k}\underline{\Pi}_{B})^{2}(I+(t_{k}\underline{\Pi}_{B})^{2})^{-1} -   
    it_{k}\underline{\Pi}_{B}(I+(t_{k}\underline{\Pi}_{B})^{2})^{-1}.
\end{equation*}
The R-bisectoriality of $\underline{\Pi}_{B}$ will thus follow once we have proven that
\begin{equation}
\label{eq:halfres}
  \mathbb{E}\BNorm{\sum_k \radem_{k} it_{k}\underline{\Pi}_{B}(I+(t_{k}\underline{\Pi}_{B})^{2})^{-1}u_{k}}{p}
  \lesssim \mathbb{E}\BNorm{\sum_k \radem_{k}u_{k}}{p} \qquad\text{and}
\end{equation}
\begin{equation}
\label{eq:otherhalfres}
  \mathbb{E}\BNorm{\sum_k \radem_{k} (t_{k}\underline{\Pi}_{B})^{2}(I+(t_{k}\underline{\Pi}_{B})^{2})^{-1}u_{k}}{p}
  \lesssim \mathbb{E}\BNorm{\sum_k \radem_{k}u_{k}}{p}. \phantom{\qquad\text{and}}
\end{equation}

To do so we note that, since $\Pi_{B}$ and $\underline{\Pi}_{B}$ are Hodge-Dirac operators with variable coefficients, we have:
\begin{equation*}
\begin{split}
\Gamma B_{1} (I+(t_{k}\underline{\Pi}_{B})^{2})^{-1}\underline{\mathbb{P}}_{1} &= 
\Gamma  (I+(t_{k}\Pi_{B})^{2})^{-1}B_{1}\underline{\mathbb{P}}_{1},\\
\rGamma (I+(t_{k}\underline{\Pi}_{B})^{2})^{-1} B_{2} \mathbb{P}_{1} &= 
\rGamma B_{2}  (I+(t_{k}\Pi_{B})^{2})^{-1}\mathbb{P}_{1}.
\end{split}
\end{equation*}
We can now proceed with the estimates, using the R-bisectoriality of $\Pi_{B}$.
\begin{equation*}
\begin{split}
    \mathbb{E}\BNorm{\sum_k \radem_{k}it_{k}\underline{\Pi}_{B}
       (I+(t_{k}\underline{\Pi}_{B})^{2})^{-1}\underline{\mathbb{P}}_{1}u_{k}}{p} 
& = \mathbb{E}\BNorm{\sum_k \radem_{k}it_{k}B_{2}\Gamma B_{1}
        (I+(t_{k}\underline{\Pi}_{B})^{2})^{-1}\underline{\mathbb{P}}_{1}u_{k}}{p}, \\
& = \mathbb{E}\BNorm{\sum_k \radem_{k}it_{k}B_{2}\Gamma
        (I+(t_{k}\Pi_{B})^{2})^{-1}B_{1}\underline{\mathbb{P}}_{1}u_{k}}{p}, \\
& = \mathbb{E}\BNorm{\sum_k \radem_{k}it_{k}B_{2}\Pi_{B}
        (I+(t_{k}\Pi_{B})^{2})^{-1}B_{1}\underline{\mathbb{P}}_{1}u_{k}}{p}, \\
& \lesssim \mathbb{E}\BNorm{\sum_k \radem_{k}B_{1}\underline{\mathbb{P}}_{1}u_{k}}{p}
  \lesssim \mathbb{E}\BNorm{\sum_k \radem_{k}u_{k}}{p}. 
 \end{split}
\end{equation*}
Introducing $v_{k}$ such that $B_{2}\mathbb{P}_{1}v_{k} = \underline{\mathbb{P}}_{2}u_{k}$, we also get:
\begin{equation*}
\begin{split}
    \mathbb{E}\BNorm{\sum_k \radem_{k}it_{k}\underline{\Pi}_{B}
        (I+(t_{k}\underline{\Pi}_{B})^{2})^{-1}\underline{\mathbb{P}}_{2}u_{k}}{p} 
& = \mathbb{E}\BNorm{\sum_k \radem_{k}it_{k}\rGamma
        (I+(t_{k}\underline{\Pi}_{B})^{2})^{-1}B_{2}\mathbb{P}_{1}v_{k}}{p}, \\
& = \mathbb{E}\BNorm{\sum_k \radem_{k}it_{k}\rGamma B_{2}
        (I+(t_{k}\Pi_{B})^{2})^{-1}\mathbb{P}_{1}v_{k}}{p}, \\
& \lesssim \mathbb{E}\BNorm{\sum_k \radem_{k}it_{k}\Pi_{B}
        (I+(t_{k}\Pi_{B})^{2})^{-1}\mathbb{P}_{1}v_{k}}{p}, \\
&\lesssim \mathbb{E}\BNorm{\sum_k \radem_{k}\mathbb{P}_{1}v_{k}}{p}
 \lesssim  \mathbb{E}\BNorm{\sum_k \radem_{k}u_{k}}{p}. 
 \end{split}
\end{equation*}
The estimate (\ref{eq:otherhalfres}) is proven in the same way. 
\end{proof}

 
\begin{lemma}\label{lem:pertSplit}
 
If a Banach space splits as \(X=X_0\oplus X_1=\proj_{0}X\oplus \proj_{1}X\), then there is \(\delta>0\) such that for all \(T\in\bddlin(X_{1},X)\) with \(\Norm{T}{}<\delta\), it also splits as \(X=X_0\oplus (I-T)X_1 = \tilde\proj_{0}X\oplus \tilde\proj_{1}X\), with   $\Norm{\tilde\proj_{0}-\proj_{0}}{}+\Norm{\tilde\proj_{1}-\proj_{1}}{} \lesssim \delta.$  
\end{lemma}

\begin{proof}
  For \(\delta:=\frac{1}{2\Norm{\proj_1}{}}\),   \(I-T\proj_1\) is invertible. Let \(U:=(I-T\proj_1)^{-1}\), and observe the identity
\begin{equation*}
  U=I+UT\proj_1.
\end{equation*}
Define the operators
\begin{equation*}
  \tilde\proj_0:=\proj_0 U,\qquad\tilde\proj_1:= (I-T\proj_1) \proj_1 U.
\end{equation*}
Then \(\range(\tilde\proj_0)=X_0\), \(\range(\tilde\proj_1)=(I-T) X_1\), and
\begin{equation*}
  \tilde\proj_0+\tilde\proj_1=(I-\proj_1+ (I-T\proj_1) \proj_1)U=(I-T\proj_1)U=I.
\end{equation*}
It remains to show that \(\tilde\proj_0\) (and then also \(\tilde\proj_1\)) is a projection. This follows from
\begin{equation*}
  \tilde\proj_0   \tilde\proj_0 =  \proj_0 U\proj_0 U
  =\proj_0(I+UT\proj_1)\proj_0 U
  =\proj_0 U =  \tilde\proj_0,
\end{equation*}
where we used \(\proj_1\proj_0=0\) and \(\proj_0^2=I\).
 
Since $\tilde\proj_{0}-\proj_{0} = \proj_{0}UT\proj_{1}$,  $\tilde\proj_{1}-\proj_{1} = \proj_{1}UT\proj_{1}-T\proj_{1}U$,
and $\|U\|\leq  2 $, we also have that   $\Norm{\tilde\proj_{0}-\proj_{0}}{}+\Norm{\tilde\proj_{1}-\proj_{1}}{} \lesssim \delta.$  
 \end{proof}


\begin{proposition}\label{prop:Hodgepert}
Let \(p\in(1,\infty)\), and \(\Pi_{A}\) be a Hodge-Dirac operator with variable coefficients in $L^p$ 
which Hodge-decomposes $L^p$.
There exists \(\delta>0\) such that, if \(\Pi_{B}\) is another Hodge-Dirac operator with variable coefficients in $L^p$ with   $\Norm{B_{1}-A_{1}}{\infty}+\Norm{B_{2}-A_{2}}{\infty} < \delta$, then
\(\Pi_{B}\) Hodge-decomposes \(L^p\).
Moreover, the associated Hodge-projections satisfy
\begin{equation*}
  \Norm{\proj^A_0-\proj^B_0}{}
  +\Norm{\proj^A_{\Gamma}-\proj^B_{\Gamma}}{}
  +\Norm{\proj^A_{\rGammaSmall_A}-\proj^B_{\rGammaSmall_B}}{}
  \lesssim\delta.
\end{equation*} 
\end{proposition}


\begin{proof}
Consider the condition \eqref{eq:HodgeAss} equivalent to the Hodge-decomposition. 
Let us define $T_{1} \in \mathscr{L}(\overline{\range_p(\rGamma_{A})},L^{p})$ by $T_{1}A_{1}\rGamma u := (A_{1}-B_{1})\rGamma u$   which, by \eqref{coercond}, gives a well-defined operator of norm $\Norm{T_1}{}\lesssim\delta$, and we have $B_1\rGamma=(I-T_1)A_1\rGamma$.

On the dual side, we define $\tilde{T}_2\in\bddlin(\overline{\range_{p'}(\rGamma_A^*)},L^{p'})$ by $\tilde{T}_2 A_2^*\rGamma^* v=(A_2^*-B_2^*)\rGamma v$, which is similarly well-defined and satisfies $\Norm{\tilde{T}_2}{}\lesssim\delta$. Let then $T_2:=(\tilde{T}_2\proj_{\rGammaSmall_A^*})^*\in\bddlin(L^p)$, where $\proj_{\rGammaSmall_A^*}$ is the projection in $L^{p'}$ associated to the decomposition in \eqref{eq:HodgeAss}. By duality, it follows that $\range_p(B_2-A_2(I-T_2))\subseteq\kernel_p(\rGamma)$, which means that $\rGamma B_2=\rGamma A_2(I-T_2)$. Since the operators $I-T_2$ and $I-T_1\proj_{\rGammaSmall_A}$ are invertible for $\delta$ small enough,  
we then have that
\begin{equation*}
  \overline{\range(\rGamma_B)}
  =\overline{\range((I-T_1)\rGamma_A (I-T_2) )}
  =(I-T_1)\overline{\range(\rGamma_A)}.
\end{equation*}

Similarly, with \(B_2^*\rGamma^*=(I-T_3)A_2^*\rGamma^*\) and \(\rGamma^* B_1^*=\rGamma^* A_1^* (I-T_4) \), there holds \(\overline{\range(\rGamma_B^*)}=(I-T_3)\overline{\range(\rGamma_A^*)}\). Hence the claim follows from two applications of Lemma~\ref{lem:pertSplit} with \((X_0,X_1,T)=(\kernel_{p}(\Gamma),\overline{\range(\rGamma_A)},T_1)\) and \((X_0,X_1,T)=(\kernel_{p'}(\Gamma^*),\overline{\range(\rGamma_A^*)},T_3)\)
\end{proof}

By \eqref{coercond}, the restriction $A_1:\overline{\range_p(\rGamma)}\to\overline{\range_p(\rGamma_A)}$ is an isomorphism, and we denote by $A_1^{-1}$ its inverse. Thus the operator $A_1^{-1}\proj_{\rGammaSmall_A}$ is well-defined. We shall also need to perturb such operators:

\begin{corollary}\label{cor:Hodgepert}
Under the assumptions of Proposition~\ref{prop:Hodgepert}, there also holds
\begin{equation*}
  \Norm{A_1^{-1}\proj^A_{\rGammaSmall_A}-B_1^{-1}\proj^B_{\rGammaSmall_B}}{}
  \lesssim\delta.
\end{equation*}
\end{corollary}

\begin{proof}
We use the same notation as in the proof of Proposition~\ref{prop:Hodgepert} and observe from the identity $B_1\rGamma=(I-T_1)A_1\rGamma$ that $I-T_1:\overline{\range_p(\rGamma_A)}\to B_1\overline{\range_p(\rGamma)}$, and
\begin{equation}\label{eq:invHodge}
  B_1^{-1}(I-T_1)\proj^A_{\rGammaSmall_A}=A_1^{-1}\proj^A_{\rGammaSmall_A}.
\end{equation}
We further recall from the proof of Lemma~\ref{lem:pertSplit} that the projection $\proj^B_{\rGammaSmall_B}$ related to the splitting $L^p=\kernel_p(\Gamma)\oplus\overline{\range_p(\rGamma_B)}$, where $\overline{\range_p(\rGamma_B)}=(I-T_1)\overline{\range_p(\rGamma_A)}$, is given by
\begin{equation*}
  \proj^B_{\rGammaSmall_B}=(I-T_1)\proj^A_{\rGammaSmall_A}(I-T_1\proj^A_{\rGammaSmall_A})^{-1}.
\end{equation*}
Combining this with \eqref{eq:invHodge} shows that
\begin{equation*}
\begin{split}
  B_1^{-1}\proj^B_{\rGammaSmall_B}
  &=A_1^{-1}\proj^A_{\rGammaSmall_A}(I-T_1\proj^A_{\rGammaSmall_A})^{-1} \\
  &=A_1^{-1}\proj^A_{\rGammaSmall_A}
   -(A_1^{-1}\proj^A_{\rGammaSmall_A})T_1\proj^A_{\rGammaSmall_A}
     (I-T_1\proj^A_{\rGammaSmall_A})^{-1},
\end{split}
\end{equation*}
which proves the claim, since the second term contains the factor $T_1$ of norm $\Norm{T_1}{}\lesssim\delta$.
\end{proof}

\section{Functional calculus}

In this section we collect some general facts about the functional calculus of bisectorial operators   in reflexive Banach spaces. We provide, in the context of the discrete randomized quadratic estimates required in \cite{hmp}, versions of results originally obtained in \cite{cdmy}.
Lemma \ref{lem:identity} can be seen as a discrete Calder\'on reproducing formula, and Lemma \ref{lem:Rbisec} as a Schur estimate, while Propositions \ref{prop:fund} and \ref{prop:funddual} express the fundamental link between functional calculus and square function estimates. Such results are not new, and have been developed from \cite{cdmy} by various authors, most notably Kalton and Weis  (cf.~\cite{kw}). Here we hope, however, to provide   simpler versions of both the statements and the proofs of these facts.

Let $\mathcal{A}$ denote a bisectorial operator in   a reflexive Banach space,   with angle $\omega$, and let $\theta \in (\omega,\frac{\pi}{2})$. We use the following notations.
$$
r(\mathcal{A}) = (I+i\mathcal{A})^{-1}, \quad p(\mathcal{A})=r(\mathcal{A})r(-\mathcal{A}) = (I+\mathcal{A}^{2})^{-1}, \quad
q(\mathcal{A})=\frac{i}{2}(r(\mathcal{A})-r(-\mathcal{A})) = \frac{\mathcal{A}}{I+\mathcal{A}^{2}}. $$

\begin{lemma}\label{lem:identity}
The following series converges in the strong operator topology:
\begin{equation*}
  \frac{3}{2}\sum_{k\in\Z}q(2^{k}\mathcal{A})q(2^{k+1}\mathcal{A})=\proj_{\overline{\range(\mathcal{A})}}.
\end{equation*}
\end{lemma}

\begin{proof}
Observe first that \(p(t\mathcal{A})\to\proj_{\kernel(\mathcal{A})}\) as \(t\to\infty\) and \(p(t\mathcal{A})\to I\) as \(t\to 0\). Hence
\begin{equation*}\begin{split}
  \proj_{\overline{\range(\mathcal{A})}}
  &=I-\proj_{\kernel(\mathcal{A})}
  =\sum_k(p(2^{k}\mathcal{A})-p(2^{k+1}\mathcal{A}) )\\
  &=\sum_k p(2^{k}\mathcal{A})[(I+2^{2(k+1)}\mathcal{A}^{2})-(I+2^{2k}\mathcal{A}^{2})]p(2^{k+1}\mathcal{A}) \\
  &=\frac{3}{2}\sum_k p(2^{k}\mathcal{A})(2^{k}\mathcal{A})(2^{k+1}\mathcal{A})p(2^{k+1}\mathcal{A})
  =\frac{3}{2}\sum_k q(2^{k}\mathcal{A})q(2^{k+1}\mathcal{A}),
  \end{split}\end{equation*}
as we wanted to show.
\end{proof}

\begin{lemma}\label{lem:Rbisec}
Let $\mathcal{A}$ be \(R\)-bisectorial and let \(\eta(x):=\min\{x,1/x\}(1+\log\max\{x,1/x\})\). Then the set
\begin{equation*}
 \{\eta(s/t)^{-1}q(t\mathcal{A}) f(\mathcal{A})q(s\mathcal{A}):t,s>0;f\in H_0^{\infty}(S_{\theta}),\Norm{f}{\infty}\leq 1\}
\end{equation*}
is R-bounded.
\end{lemma}

\begin{proof} 
Denoting $q_t(\lambda):=q(t\lambda)$,   observe that
\begin{equation*}
  q(t\mathcal{A}) f(\mathcal{A})q(s\mathcal{A})
    =  (q_t\cdot f\cdot q_s)  (\mathcal{A})
  =\frac{1}{2\pi i}\int_{\gamma} (q_t fq_s) (\lambda)(I-\frac{1}{\lambda}\mathcal{A})^{-1}\frac{\ud\lambda}{\lambda},
\end{equation*}
where $\gamma$ denotes \(\partial S_{\theta'}\), for some $\theta' \in (\theta,\frac{\pi}{2})$, parameterized by arclength and directed anti-clockwise 
around \(S_{\theta}\), and
the resolvents \((I-\lambda^{-1}\mathcal{A})^{-1}\) belong to an \(R\)-bounded set. The operators 
\(q(t\mathcal{A}) f(\mathcal{A})q(s\mathcal{A})\) are hence in a dilation of the absolute convex hull of this \(R\)-bounded set (see \cite{kk} for information on R-boundedness techniques). To evaluate the dilation factor, observe that
\begin{equation*}
  \abs{  q_t f q_s  (\lambda)}\lesssim\Norm{f}{\infty}\frac{s\abs{\lambda}}{1+(s\abs{\lambda})^2}
   \frac{t\abs{\lambda}}{1+(t\abs{\lambda})^2},
\end{equation*}
and splitting the integral into three regions (depending on the position of $|\lambda|$ with respect
to $\min(\frac{1}{t},\frac{1}{s})$  and $\max(\frac{1}{t},\frac{1}{s})$) it follows that
\begin{equation*}
  \int_{\gamma}\abs{  q_t f q_s  (\lambda)}\frac{\abs{\ud\lambda}}{\abs{\lambda}}
  \lesssim\Norm{f}{\infty}\eta(s/t),
\end{equation*}
which implies the asserted \(R\)-bound.
\end{proof}



\begin{proposition}
\label{prop:fund}
Let \(\mathcal{A}\) be \(R\)-bisectorial (with angle $\mu$) and satisfy the   two-sided   quadratic estimate
\begin{equation}\label{eq:quadrEst}
  \Norm{u}{}
  \eqsim\Exp\BNorm{\sum_k\radem_k q(2^{k}\mathcal{A}) u}{},\qquad
  u\in\overline{\range(\mathcal{A})}.
\end{equation}
Then \(\mathcal{A}\) has a bounded \(H^{\infty}\) functional calculus (with angle $\mu$).
\end{proposition}

\begin{proof}
Suppose (\ref{eq:quadrEst}) holds. Let \(u\in\overline{\range(\mathcal{A})}\),
$\theta \in (\mu,\frac{\pi}{2})$,
 and \(f\in H^{\infty}(S_{\theta})\). Then
\begin{equation*}\begin{split}
  \Norm{f(\mathcal{A})u}{} &\lesssim\Exp\BNorm{\sum_k\radem_k q(2^{k}\mathcal{A}) f(\mathcal{A})u}{}\\
  &\eqsim\Exp\BNorm{\sum_k\radem_k q(2^{k}\mathcal{A}) f(\mathcal{A}) \sum_j q(2^{j}\mathcal{A})q(2^{j+1}\mathcal{A})u}{}  \\
  &\lesssim\sum_m\Exp\BNorm{\sum_k\radem_kq(2^{k}\mathcal{A}) f(\mathcal{A})q(2^{k+m}\mathcal{A})
  q(2^{k+m+1}\mathcal{A})u}{} \\
  &\lesssim\sum_m\eta(2^m)\Norm{f}{\infty}\BNorm{\sum_k\radem_k q(2^{k+m+1}\mathcal{A})u}{} \\
  &\lesssim\sum_m(1+\abs{m})2^{-\abs{m}}\Norm{f}{\infty}\Norm{u}{}\lesssim\Norm{f}{\infty}\Norm{u}{},
   \end{split}
\end{equation*}
where we used \(\lesssim\) from \eqref{eq:quadrEst}, Lemma~\ref{lem:identity}, the triangle inequality after relabelling \(j=k+m\), Lemma~\ref{lem:Rbisec}, and \(\gtrsim\) from \eqref{eq:quadrEst}.
\end{proof}

We also use the following variant.

\begin{proposition}
\label{prop:funddual}
Let \(\mathcal{A}\) be \(R\)-bisectorial (with angle $\mu$) and satisfy the   two quadratic estimates  
\begin{equation}
\begin{split}
  \Exp\BNorm{\sum_k\radem_k q(2^{k}\mathcal{A}) u}{X}
  &\lesssim \Norm{u}{X},\qquad  u\in X,\\
  \Exp\BNorm{\sum_k\radem_k q(2^{k}\mathcal{A}^*) v}{X^{*}}
  &\lesssim \Norm{v}{X^{*}},\qquad  v\in X^{*},
\end{split}
\end{equation}
   
Then \(\mathcal{A}\) has a bounded \(H^{\infty}\) functional calculus (with angle $\mu$).
\end{proposition}

\begin{proof}
 
 Let \(u\in X\), $v \in X^{*}$,   $\theta \in (\mu,\frac{\pi}{2})$, and \(f\in H^{\infty}(S_{\theta})\). Then
\begin{equation*}
\begin{split}
|\langle f(\mathcal{A})u,v \rangle| &\leq \sum \limits _{k} |\langle q(2^{k}\mathcal{A})f(\mathcal{A})u, q(2^{k+1}\mathcal{A}^{*})v \rangle|\\
&\lesssim \Exp\BNorm{\sum_k\radem_k q(2^{k}\mathcal{A}) f(\mathcal{A})u}{X}
\Exp\BNorm{\sum_k\radem_k q(2^{k}\mathcal{A}^*)v}{X^{*}} \lesssim \Exp\BNorm{\sum_k\radem_k q(2^{k}\mathcal{A}) f(\mathcal{A})u}{X} \|v\|_{X^{*}}.
\end{split}
\end{equation*}
The proof is then concluded as in Proposition \ref{prop:fund}.
\end{proof}

\section{$L^p$ theory for operators with variable coefficients}
\label{sect:proofs}


In this section, we give the proofs of the results stated in Subsection \ref{ss:pert}, and some variations.
The core result, Theorem \ref{thm:char}, is a generalization of \cite[Theorem 3.1]{hmp}. The key ingredients of the proof are contained in \cite{hmp}.
Here we indicate where \cite{hmp} needs to be modified, using the results from the preceding sections.
  Let us recall that $\Pi$ denotes an Hodge-Dirac operator with constant coefficients as defined in Definition \ref{def:unpert}, and that 
$\Pi_{B}$ denotes a Hodge-Dirac operator with variable coefficients as defined in Definition \ref{def:pert}.  
We also use the following notation.
\begin{equation*}\begin{split}
  R_t ^{B} &:=(I+it\Pi_{B})^{-1} =r(t\Pi_{B}),\\
  P_t ^{B} &:=(I+t^2\Pi_{B} ^2)^{-1}= p(t\Pi_{B}),\\
  Q_t ^{B} &:=t\Pi_{B} P_t ^{B}= q(t\Pi_{B}),
  \end{split}\end{equation*}
  and denote by $R_{t},P_{t},Q_{t}$ the corresponding functions of $\Pi$.  


\begin{theorem}\label{thm:char}
Let $1\leq p_1<p_2\leq \infty$, 
and let \(\Pi_{B}\) be an R-bisectorial 
Hodge-Dirac operator with variable coefficients in $L^p$ which Hodge-decomposes $L^p$ for all \(p \in (p_{1},p_{2})\). Then
\begin{equation}\label{eq:onesidedQE}
  \Exp\BNorm{\sum_k\radem_k Q_{2^k}^B u}{p}
  \lesssim  \Norm{u}{p},\qquad
  u\in\overline{\range_{p}(\Gamma)},
\end{equation}
and
\begin{equation}\label{eq:dualonesidedQE}
  \Exp\BNorm{\sum_k\radem_k (Q_{2^k}^B)^* v}{p'}
  \lesssim  \Norm{v}{p'},\qquad
  v\in\overline{\range_{p'}(\Gamma^*)}.
\end{equation}
\end{theorem}

\begin{proof}
We first prove \eqref{eq:onesidedQE}.  
Let us recall the following notation from \cite{hmp}.
Let
\begin{equation*}
  \triangle := \bigcup_{k \in \Z}\triangle_{2^k},\qquad
  \triangle_{2^k}:=\big\{2^k([0,1)^n+m):m\in\Z^n\big\}.
\end{equation*}
denote a system of dyadic cubes, and
\begin{equation*}
  A_{2^k} u(x):=\ave{u}_Q 
  :=\frac{1}{\abs{Q}}\int_Q u(y)\ud y,\qquad
  x\in Q\in\triangle_{2^k}.
\end{equation*}
be the corresponding conditional expectation projections.
Let
\begin{equation}\label{eq:prPart}
  \gamma_{2^k}(x)w:=Q_{2^k}^{B}(w)(x)
  :=\sum_{Q\in\triangle_{2^k}}Q_{2^k}^{B}(w1_{Q})(x),
  \quad x\in\R^n,\ w\in \C^{N}.
\end{equation}
denote the   \emph{principal part}   of $Q_{2^{k}}^{B}$,   which we also identify with the corresponding pointwise multiplication operator.

The proof of (\ref{eq:onesidedQE}) now divides into the following four estimates: 
\begin{equation}
\label{eq:highfreq}
\Exp\BNorm{\sum_k\radem_k Q_{2^k}^B (I-P_{2^k})u}{p}
\lesssim  \Norm{u}{p}
,\qquad
  u\in\overline{\range_{p}(\Gamma)}.
\end{equation}
\begin{equation}
\label{eq:Qt1}
\Exp\BNorm{\sum_k\radem_k (Q_{2^k}^B - \gamma_{2^k}A_{2^k})P_{2^k}u}{p}
\lesssim  \Norm{u}{p}
,\qquad
  u\in\overline{\range_{p}(\Gamma)}.
\end{equation}
\begin{equation}
\label{eq:dyadicvscont}
\Exp\BNorm{\sum_k\radem_k  \gamma_{2^k}A_{2^k}(I-P_{2^k})u}{p}
\lesssim  \Norm{u}{p}
,\qquad
  u\in\overline{\range_{p}(\Gamma)}.
\end{equation}
\begin{equation}
\label{eq:carleson}
\Exp\BNorm{\sum_k\radem_k \gamma_{2^k}A_{2^k}u}{p}
\lesssim  \Norm{u}{p}
,\qquad
  u\in\overline{\range_{p}(\Gamma)}.
\end{equation}

  Inequality   (\ref{eq:highfreq}) follows from the fact that $\Pi_B$ Hodge-decomposes $L^p$, and   from   the R-bisectorial\-ity of $\Pi_B$:
  as in \cite[Lemma 6.3]{hmp},   denoting by   $\proj_{\rGammaSmall_B}$   the projection onto $\overline{\range_p(\rGamma_{B})}$   in the Hodge-decomposition,   we have
$$
Q_{2^{k}} ^{B} (I-P_{2^k})u = (I-P_{2^k} ^{B})\proj_{\rGammaSmall_B}Q_{2^k}u
,\qquad
  u\in\overline{\range_{p}(\Gamma)},
$$
and
\begin{equation}\label{eq:highFreqProof}
  \Exp\BNorm{\sum_k\radem_k (I-P_{2^k} ^{B})\proj_{\rGammaSmall_B}Q_{2^k}u}{p}
  \lesssim \Exp\BNorm{\sum_k\radem_k Q_{2^k}u}{p}
  \lesssim  \Norm{u}{p},\qquad
  u\in  L^p, 
\end{equation}
where the last inequality follows from Theorem \ref{thm:unpert}.

To prove \eqref{eq:Qt1} and \eqref{eq:dyadicvscont} we first note that
  commutators of the form $[\eta I, \Gamma]$ are multiplication operators by an $L^{\infty}$ function bounded by 
$\Norm{\nabla \eta}{\infty}$. Then the R-bisectoriality of $\Pi_{B}$ and \cite[Proposition 6.4]{hmp} give the following \emph{off-diagonal R-bounds}:  
for every $M \in \N$, every Borel subsets $E_k,\
F_k\subset\R^n$, every $u_k\in L^p$, and every
$(t_k)_{k\in\Z}\subseteq\{2^j\}_{j\in\Z}$ such that
$\dist(E_k,F_k)/t_k>\varrho$ for some $\varrho>0$ and all $k\in\Z$,
there holds
\begin{equation}\label{eq:offDiagonal}\begin{split}
  \Exp &\BNorm{\sum_k\radem_k1_{E_k}
     Q^B_{  t_k } 1_{F_k}u_k}{p} 
  \lesssim(1+\varrho)^{-M}\Exp\BNorm{\sum_k\radem_k1_{F_k}u_k}{p}.
\end{split}\end{equation}
This, in turn, gives that   the family $(\gamma_{2^{k}}A_{2^k})_{k\in\Z}$   is R-bounded exactly as in \cite[Lemma 6.5]{hmp}.

Now let us prove (\ref{eq:Qt1}). This is similar to \cite[Lemma 6.6]{hmp} but, since a few modifications need to be made, we include the proof.
Letting $v_{k} = P_{2^k}u$, and using the off-diagonal R-bounds, we have:
\begin{equation}
\begin{split}
  \Exp &\BNorm{\sum_k\radem_k(Q_{2^k} ^{B}-\gamma_{2^k}A_{2^k})v_k}{p} \\
  &=\Exp\BNorm{\sum_k\radem_k\sum_{Q\in\triangle_{2^k}}1_Q
     Q_{2^k} ^{B}\big(v_k-\ave{v_k}_Q)}{p} \\
  &\leq\sum_{m\in\Z^n}\Exp\BNorm{\sum_k\radem_k\sum_{Q\in\triangle_{2^k}}
     1_Q Q_{2^k} ^{B}\big(1_{Q-2^k m}(v_k-\ave{v_k}_Q)\big)}{p} \\
  &\lesssim\sum_{m\in\Z^n}(1+\abs{m})^{-M}\,
   \Exp\BNorm{\sum_k\radem_k\sum_{Q\in\triangle_{2^k}}
     1_Q(v_k-\ave{v_k}_{Q+2^k m})}{p}.
\end{split}\end{equation}
A version of Poincar\'e's inequality \cite[Proposition 4.1]{hmp} allows to majorize the last factor by  
\begin{equation*}
  \int_{[-1,1]^n}\int_0^1
    \Exp\BNorm{\sum_k\radem_k 2^k(m+z) \cdot \nabla\,\tau_{t2^k(m+z)}
     P_{2^k} u}{p}
    \ud t \ud z. \\
\end{equation*} 
We then use Propositions \ref{prop:coer} and \ref{prop:transQuadEst} to estimate this term by
\begin{equation*}
  (1+|m|)(1+\log_{+}|m|)\|u\|_{p},
\end{equation*}
and the proof is thus completed by picking $M$ large enough.

To prove \eqref{eq:dyadicvscont}, we first use the R-boundedness of $(\gamma_{2^k}A_{2^k})$ and the idempotence of $A_{2^k}$. We thus have to show   that   
\begin{equation*}
  \Exp\BNorm{\sum_k\radem_k A_{2^k}(I-P_{2^k})u}{p}
  \lesssim  \Norm{u}{p},\qquad u\in\overline{\range_p(\Gamma)}.
\end{equation*} 
This is essentially like \cite[Proposition 5.5]{hmp}. We indicate the beginning of the argument, where the operator-theoretic Lemma~\ref{lem:identity} replaces a Fourier multiplier trick used in \cite{hmp}. Indeed, with $u\in\overline{\range_p(\Gamma)}\subseteq\overline{\range_p(\Pi)}$, we have
\begin{equation*}
\begin{split}
  \Exp\BNorm{\sum_k\radem_k A_{2^k}(I-P_{2^k})u}{p}
  &\eqsim\Exp\BNorm{\sum_k\radem_k A_{2^k}(I-P_{2^k})\sum_j Q_{2^j}Q_{2^{j+1}}u}{p} \\
  &\leq\sum_m\Exp\BNorm{\sum_k\radem_k \big(A_{2^k}(I-P_{2^k})Q_{2^{k+m}}\big)Q_{2^{k+m+1}}u}{p}.
\end{split}
\end{equation*} 
Thanks to the second estimate in \eqref{eq:highFreqProof}, it suffices to show that the operator family
\begin{equation*}
  \big\{A_{2^k}(I-P_{2^k})Q_{2^{k+m}}:k\in\Z\big\}\subset\bddlin(L^p)
\end{equation*}
is R-bounded with R-bound $C2^{-\delta\abs{m}}$ for some $\delta>0$. This is done by repeating the argument of \cite[Proposition 5.5]{hmp}.

Finally, the proof of (\ref{eq:carleson}) is done exactly as in   \cite[Theorem 8.2 and Proposition 9.1]{hmp}.
Notice that this last part is the only place where we need the assumptions for $p$ in an open interval $(p_1,p_2)$; all the other estimates work for a fixed value of $p$.

This completes the proof of (\ref{eq:onesidedQE}).

Let us turn our attention to (\ref{eq:dualonesidedQE}).  By Lemma \ref{lem:dual}, we have that 
${\Pi_B}^* =\Gamma^* +{B_2}^*{\rGamma}^*{B_1}^*$ is a Hodge-Dirac operator with variable coefficients in $L^{p'}$ which Hodge-decomposes $L^{p'}$. It is also R-bisectorial, as this property is preserved under duality (see \cite[Lemma 3.1]{kw}). So the proof of (\ref{eq:onesidedQE}) adapts to give the quadratic estimate (\ref{eq:dualonesidedQE}) involving $(Q^B_t)^*= t{\Pi_B}^*(I+(t{\Pi_B}^*)^2)^{-1}$.
\end{proof} 

We now prove Theorem \ref{thm:base} as a corollary.

\begin{corollary}\label{cor:char}
Let \(1\leq p_{1}<p_{2} \leq \infty\),   $\mu\in(\omega,\pi/2)$,  
and let \(\Pi_{B}\) be a Hodge-Dirac operator with variable coefficients in $L^p$ which Hodge-decomposes \(L^{p}\) for all \(p \in (p_{1},p_{2})\).
Assume also that $\underline{\Pi}_B$ is a Hodge-Dirac operator with variable coefficients in $L^p$.
Then \(\Pi_{B}\) has a bounded \(H^{\infty}\) functional calculus (with angle $\mu$) in \(L^{p}(\R^{n};\C^{N})\) for all \(p \in (p_{1},p_{2})\) if and only if it is \(R\)-bisectorial (with angle $\mu$) in \(L^{p}(\R^{n};\C^{N})\)  for all \(p \in (p_{1},p_{2})\).
\end{corollary}

\begin{proof}
The fact that a bounded $H^{\infty}$ functional calculus implies R-bisectoriality is a general property (see Remark \ref{fcRsect}).
To prove the other direction,   assume   that $\Pi_{B}$ is R-bisectorial on \(L^{p}(\R^{n};\C^{N})\)  for all \(p \in (p_{1},p_{2})\).
By Theorem \ref{thm:char}, we have that
\begin{equation}\label{eq:QEonRGamma}
  \Exp\BNorm{\sum_k\radem_k Q_{2^k}^B u}{p}
  \lesssim  \Norm{u}{p},\qquad
   u\in\overline{\range_{p}(\Gamma)}.
\end{equation}
Moreover, since $\underline{\Pi}_B$ also satisfies the assumptions of Theorem \ref{thm:char} 
(using Lemma \ref{lem:tilde}), we have that
\begin{equation}\label{eq:underlineQE}
  \Exp\BNorm{\sum_k\radem_k 2^{k}\underline{\Pi}_B(I+(2^{k}\underline{\Pi}_B)^{2})^{-1} u}{p}
  \lesssim  \Norm{u}{p},\qquad
  u\in\overline{\range_{p}(\rGamma)}.
\end{equation}
 
For $u\in\overline{\range_p(\rGamma)}$, there holds
\begin{equation*}
\begin{split}
  & 2^{k}\underline{\Pi}_B(I+(2^{k}\underline{\Pi}_B)^{2})^{-1} u
  =2^{k}B_{2}\Gamma B_{1}(I+(2^{k}\underline{\Pi}_B)^{2})^{-1} u \\
  & =2^k B_2\Gamma (I+(2^{k}\Pi_{B})^{2})^{-1} B_{1} u
  =2^k B_2\Pi_B (I+(2^{k}\Pi_{B})^{2})^{-1} B_{1} u.
\end{split}
\end{equation*}
Thus by \eqref{coercond}, the estimate \eqref{eq:underlineQE} implies
\begin{equation}\label{eq:QEonRrGammaB}
  \Exp\BNorm{\sum_k\radem_k Q_{2^{k}}^B u}{p}
  \lesssim  \Norm{u}{p},\qquad
  u\in\overline{\range_{p}(\rGamma_B)}.
\end{equation}
Combining \eqref{eq:QEonRGamma} and \eqref{eq:QEonRrGammaB} with the Hodge-decomposition and the obvious fact that $Q_{2^k}^B$ annihilates $\kernel_p(\Pi_B)$, we arrive at  
\begin{equation*}
  \Exp\BNorm{\sum_k\radem_k Q_{2^k}^B u}{p}
  \lesssim  \Norm{u}{p},\qquad
  u\in  L^p. 
\end{equation*}

In the same way,  one gets   the dual estimate
\begin{equation*}
  \Exp\BNorm{\sum_k\radem_k (Q_{2^k}^B)^{*} u}{p'}
  \lesssim  \Norm{u}{p'},\qquad u\in L^{p'}, 
\end{equation*}
where $p'$ denotes the   conjugate   exponent of $p$.
The   functional calculus   then follows from Proposition \ref{prop:funddual}.
\end{proof}

\begin{corollary}\label{cor:pert}
Let \(1\leq p_{1}<p_{2} \leq \infty\), and let \(\Pi_{A}\) be a Hodge-Dirac operator with variable coefficients, which is R-bisectorial in $L^p$ and Hodge-decomposes \(L^p\) for all \(p \in (p_{1},p_{2})\).
Then for each $p\in(p_1,p_2)$, there exists $\delta=\delta_p>0$ such that, if \(\Pi_{B}\) and $\underline{\Pi}_B$  are Hodge-Dirac operators with variable coefficients such that $\Norm{B_{1}-A_{1}}{\infty}+\Norm{B_{2}-A_{2}}{\infty} < \delta$,   then
\(\Pi_{B}\) has an \(H^{\infty}\) functional calculus in \(L^p\) and Hodge-decomposes \(L^p\).
\end{corollary}

\begin{proof} Let $p\in (p_{1},p_{2})$. By Proposition \ref{prop:Hodgepert}, we have that, for $\delta$ small enough, $\Pi_B$ Hodge-decomposes $L^p$. We need to show that $\Pi_B$ is R-bisectorial in $L^p$ provided $\delta$ is sufficiently small.
As in the proof of Proposition~\ref{prop:Hodgepert}, let $T_1\in\bddlin(\overline{\range_p(\rGamma_A)},L^p)$ and $T_2\in\bddlin(L^p)$ be operators of norm $\Norm{T_i}{}\lesssim\delta$ such that
\begin{equation*}
  B_1\rGamma=(I-T_1)A_1\rGamma,\qquad
  \rGamma B_2=\rGamma A_2(I-T_2).
\end{equation*}
Then
\begin{equation*}
  \Pi_B
  =\Gamma+B_1\rGamma B_2
  =\Gamma+(I-T_1)A_1\rGamma A_2(I-T_2)
  =(I-T_1\proj_{\rGammaSmall_A})\Pi_A(I-\proj_{\Gamma}T_2),
\end{equation*}
where $\proj_{\Gamma}$ and $\proj_{\rGammaSmall_A}$ are the Hodge-projections associated to $\Pi_A$, onto $\overline{\range_p(\Gamma)}$ and $\overline{\range_p(\rGamma_A)}$, respectively. Hence
\begin{equation*}
\begin{split}
  I+it\Pi_B
  &=(I-T_1\proj_{\rGammaSmall_A})(I+it\Pi_A)(I-\proj_{\Gamma}T_2)
   +(T_1\proj_{\rGammaSmall_A}+\proj_{\Gamma}T_2) \\
  &=(I-T_1\proj_{\rGammaSmall_A})(I+it\Pi_A)(I-\proj_{\Gamma}T_2) \\
  &\qquad\times\big[I+(I-\proj_{\Gamma}T_2)^{-1}(I+it\Pi_A)^{-1}(I-T_1\proj_{\rGammaSmall_A})^{-1}
       (T_1\proj_{\rGammaSmall_A}+\proj_{\Gamma}T_2)],
\end{split}
\end{equation*}
where the inverses involving $T_i$ exist for $\delta$ small enough. Hence $(I+it\Pi_B)^{-1}$ can be expressed as a Neumann series involving powers of the operators $(I+it\Pi_A)^{-1}$, which are R-bounded by assumption, times powers of fixed bounded operators, including $T_1\proj_{\rGammaSmall_A}+\proj_{\Gamma}T_2$ which has norm at most $C\delta$. For $\delta$ small enough, the R-boundedness of $(I+it\Pi_B)^{-1}$ follows from this representation.

Given $\varepsilon \in (p-p_{1},p_{2}-p)$ we thus have that there exists $\delta_{p,\varepsilon}$ such that $\Pi_{B}$ is R-bisectorial in $L^{p-\varepsilon}$ and in $L^{p+\varepsilon}$, and Hodge decomposes $L^{p-\varepsilon}$ and $L^{p+\varepsilon}$.
By interpolation (cf. Remark~\ref{rem:interpolation}), $\Pi_{B}$ is R-bisectorial in $L^{\tilde{p}}$ and Hodge decomposes $L^{\tilde{p}}$ for all $\tilde{p} \in (p-\varepsilon,p+\varepsilon)$.

Now the conditions of Corollary~\ref{cor:char} are verified for the operators $\Pi_B$ and $\underline{\Pi}_B$, so the mentioned result implies that $\Pi_B$ has a bounded $H^{\infty}$ functional calculus in $L^p$, as claimed.
\end{proof}
 
With potential applications to boundary value problems in mind (see \cite{aam}), we conclude this section with the following special case. This proof is essentially the same as in the $L^2$ case \cite[Theorem 3.1]{akm}.  

\begin{corollary}\label{cor:DB}
  Let \(1\leq p_{1}<p_{2} \leq \infty\).  
Let $D = -i\sum_{j=1}^n\hat{D}_j\partial_j$ be a first order differential operator with $\hat{D}_{j} \in \mathscr{L}(\C^{N})$,
and $A \in L^{\infty}(\R^{n};\mathscr{L}(\C^{N}))$ be such that
\begin{equation*}\tag{H1}\label{H1}
\begin{split}
  \abs{\xi}\abs{e}\lesssim\abs{\hat{D}(\xi)e}\quad &\text{for all}\quad e\in\range(\hat{D}(\xi)),\quad   \text{and}\\     
    \sigma(\hat{D}(\xi))\subseteq S_{\omega}\quad
  &\text{for some}\quad \omega \in (0,\frac{\pi}{2})\quad \text{and all}\quad \xi \in \R^{n},
\end{split}
\end{equation*}
\begin{equation*}\tag{H2}\label{H2}
     \Norm{u}{p} \lesssim \Norm{Au}{p}\quad \text{for all}\quad u \in \range_p(D),
\end{equation*}
\begin{equation*}\tag{H3}\label{H3}
  \Norm{u}{p'} \lesssim \Norm{A^{*}u}{p'}\quad \text{for all}\quad   u \in \range_{p'}(D^{*}),  
\end{equation*}
for all $p\in(p_1,p_2)$. Then we have the following:  
\begin{enumerate}
\item\label{it:DAchar}
  The operator   $DA$ has an $H^{\infty}$ functional calculus (with angle $\mu$) in $L^{p}$ for all $p \in (p_{1},p_{2})$ if and only if it is R-bisectorial (with angle $\mu$) in $L^{p}$ for all $p \in (p_{1},p_{2})$.
\item\label{it:DApert}
  If the equivalent conditions of (1) hold, then for each $p\in(p_1,p_2)$, there exists $\delta=\delta_p>0$ such that, if another $\tilde{A} \in L^{\infty}(\R^{n};\mathscr{L}(\C^{N}))$ satisfies $\Norm{A-\tilde{A}}{\infty}<\delta$, then $D\tilde{A}$ also has an $H^{\infty}$ functional calculus in  $L^{p}$.
\end{enumerate}
\end{corollary}

\begin{proof}
  \eqref{it:DAchar} The philosophy of the proof is to reduce the consideration of an operator of the form $DA$ to the Hodge-Dirac operator with variable coefficients $\Pi_B$, which we already understand from the previous results.  
On $\C^{N}\oplus \C^{N}$, consider   the matrices
\begin{equation*}
 \hat \Gamma_{j} := \begin{pmatrix} 0 & 0 \\ \hat D_{j} & 0\end{pmatrix},\quad
 \hat \rGamma_{j} := \begin{pmatrix} 0 & \hat D_{j} \\ 0 & 0\end{pmatrix},\quad
 B_{1} := \begin{pmatrix} A & 0 \\ 0 & 0\end{pmatrix},\quad 
 B_{2} := \begin{pmatrix} 0 & 0 \\ 0 & A\end{pmatrix}.
\end{equation*}
We define the associated differential operators $\Gamma$, $\rGamma$ and $\Pi$ acting in $L^p:=L^p(\R^n;\C^N\oplus\C^N)$ as in Subsections~\ref{ss:unpert}, and the operators
\begin{equation*}
  \Pi_B=\begin{pmatrix} 0 & ADA \\ D & 0 \end{pmatrix},\qquad
  \underline{\Pi}_B=\begin{pmatrix} 0 & D \\ ADA & 0 \end{pmatrix},
\end{equation*}
as in Subsection \ref{ss:pert}.

We claim that both $\Pi_{B}$ and $\underline{\Pi}_B$   are then Hodge-Dirac operators with variable coefficients in $L^p$ which Hodge-decompose 
$L^{p}$ for all $p \in (p_{1},p_{2})$.
  Indeed, the hypotheses \eqref{H1}, \eqref{H2} and \eqref{H3} guarantee the conditions \eqref{Pi1}, \eqref{Pi2} and \eqref{coercond}. The remaining requirements \eqref{Pi3} and \eqref{nilpcond}, as well as   the Hodge decomposition (see \cite[Lemma 3.5]{hmp}), are satisfied because of the special form of $\Pi_{B}$.

A computation shows that
\begin{equation*}
  (I+it\Pi_{B})^{-1}
  =\begin{pmatrix} I & -itADA \\ 0 & I\end{pmatrix}
   \begin{pmatrix} I & 0\\ 0 & (I+t^{2}(DA)^{2})^{-1}\end{pmatrix}
   \begin{pmatrix} I & 0 \\ -itD &   I  \end{pmatrix}.
\end{equation*}
  Assuming that $DA$ is R-bisectorial, we check that so is $\Pi_B$. This amounts to verifying the R-boundedness of the families of operators
\begin{equation*}
  (I+t^{2}(DA)^{2})^{-1},\quad
  -itADA(I+t^{2}(DA)^{2})^{-1},\quad
  -it(I+t^{2}(DA)^{2})^{-1}D,
\end{equation*}
where $t\in\R$. For the first two, this is immediate from the R-bisectoriality of $DA$ and the boundedness of $A$. For the third one, we need the stability of R-boundedness in the $L^p$ spaces under adjoints, and hypothesis \eqref{H3} which allows to reduce the adjoint $-itD^*(I+t^2(A^*D^*)^2)^{-1}$ to a function of $(DA)^*=A^*D^*$ by composing with $A^*$ from the left.

Thus, by Corollary \ref{cor:char}, $\Pi_{B}$ has an $H^{\infty}$ functional calculus on $L^{p}$ for all $p \in (p_{1},p_{2})$. But the resolvent formula above also gives
\begin{equation*}
  f(\Pi_{B})\begin{pmatrix} Au \\ u \end{pmatrix}
  = \begin{pmatrix} Af(DA)u \\ f(DA)u \end{pmatrix},
\end{equation*}
for $f \in H^{\infty}_{  0 }(S_{\theta})$ and $u \in L^p$,   and hence we find that $DA$ has an $H^{\infty}$ functional calculus, too. This completes the proof that the R-bisectoriality of $DA$ implies functional calculus. The converse direction   is a general property of the functional calculus in $L^{p} $ (see Remark \ref{fcRsect}).

  \eqref{it:DApert} We turn to the second part and assume that $DA$ is R-bisectorial. We first note that $\tilde{A}$ also satisfies the hypotheses \eqref{H2} and \eqref{H3} when $\delta$ is small enough. Indeed, for $u\in\range_p(D)$, we have $\Norm{\tilde{A}u}{p}\geq\Norm{Au}{p}-\Norm{A-\tilde{A}}{\infty}\Norm{u}{p}\geq(c-\delta)\Norm{u}{p}$, and \eqref{H3} is proven similarly. Moreover, the R-bisectoriality of $DA$ implies the same property for $D\tilde{A}$ by a Neumann series argument, as
\begin{equation*}
  I+itD\tilde{A}=I+itDA-itD(A-\tilde{A})
  =(I+itDA)[I-(I+itDA)^{-1}itD(A-\tilde{A})],
\end{equation*}
where the family $\{(I+itDA)^{-1}itD:t\in\R\}$ is R-bounded (by duality, \eqref{H3}, and the R-bisectoriality of $DA$), and the factor $\Norm{A-\tilde{A}}{\infty}<\delta$ ensures convergence for $\delta$ small enough. The $H^{\infty}$ calculus then follows from part \eqref{it:DAchar} applied to $\tilde{A}$ in place of $A$. 
\end{proof}

\section{Lipschitz estimates}

In this final section, we   prove   Lipschitz estimates of the form
\begin{equation*}
  \Norm{f(\Pi_{B})u-f(\Pi_{A})u}{p}
  \lesssim  \max_{i=1,2}\Norm{A_i-B_i}{\infty} \Norm{f}{\infty}\Norm{u}{p},
\end{equation*}
for small perturbations of the coefficient matrices involved in the Hodge-Dirac operators.
Such estimates are obtained via holomorphic dependence results for perturbations $B_{z}$ depending on a complex parameter $z$. This technique can be seen as one of the original motivations for studying the Kato problem  for operators with complex coefficients.
We start with the operators studied in Corollary \ref{cor:DB}, and then deduce similar estimates for general Hodge-Dirac operators with variable coefficients, as in \cite[Section 10.1]{aamescorial}.

Let $D$ be a first order differential operator as in Corollary \ref{cor:DB}.
Let $U$ be an open set of $\C$ and $(A_{z})_{z\in U}$ a family of multiplication operators   such that the map $z\mapsto A_z\in L^{\infty}(\R^n;\C^N)$ is holomorphic.  
Let $1\leq p_{1} <p_{2} \leq \infty$, $z_{0}\in U$, and assume that $DA_{z_{0}}$ has a bounded $H^{\infty}$ functional calculus in $L^p$ for all $p\in (p_{1},p_{2})$. By Corollary \ref{cor:DB}, for each $p\in(p_1,p_2)$, there then exists a $\delta=\delta_p>0$ 
such that $DA_{z}$ has a bounded $H^{\infty}$ functional calculus in $L^p$  for all $z \in B(z_{0},\delta)$. Moreover, we have the following.

\begin{proposition}\label{prop:holo}
For $\theta \in (\mu, \frac{\pi}{2})$ and $f \in H^{\infty}(S_{\theta})$, the function $z \mapsto f(DA_{z})$
is holomorphic on $\disc(z_{0},\delta)$.
\end{proposition}

\begin{proof}
This is entirely similar to \cite[Theorem 6.1, Theorem 6.4]{akm}. 
Letting $\tau\in\C\setminus S_{\theta}$, we have
\begin{equation*}
  \frac{d}{dz}(I+\tau DA_{z})^{-1} = -(I+\tau DA_{z})^{-1}\tau DA'_{z} (I+\tau DA_{z})^{-1}.
\end{equation*}
From \eqref{H1}, \eqref{H3} and the bisectoriality of $DA_{z}$, these operators are uniformly bounded for $z \in U$, and thus the functions $z \mapsto (I+\tau DA_{z})^{-1}$ are 
holomorphic. The result is then obtained by passing to uniform limits in the strong operator toplogy.
\end{proof}

The Lipschitz estimates now follow.

\begin{corollary}\label{cor:LipDA}
Let $1\leq p_{1}<p_{2}\leq \infty$, let $D$ be a first order differential operator, and $A\in L^{\infty}(\R^n;\C^N)$ a multiplication operator which satisfy the hypotheses \eqref{H1}, \eqref{H2} and \eqref{H3} of Corollary \ref{cor:DB}. Let moreover $DA$ be R-bisectorial.
Then, for each $p\in(p_1,p_2)$, there exists $\delta=\delta_p>0$ such that, if $\tilde{A} \in L^{\infty}(\R^n;\C^N)$ satisfies $\Norm{A-\tilde{A}}{\infty}<\delta$,
then $D\tilde{A}$ has a bounded $H^\infty$ functional calculus in $L^p$ with some angle $\omega \in (0,\frac{\pi}{2})$, and for $\theta \in (\omega,\frac{\pi}{2})$, $f\in H^{\infty}(S_{\theta})$, and $u \in L^p$ we have
\begin{equation*}
  \Norm{f(DA)u-f(D\tilde{A})u}{p} \lesssim \Norm{A-\tilde{A}}{\infty}\Norm{f}{\infty}\Norm{u}{p}.
\end{equation*}
\end{corollary}

\begin{proof} 
Let $A_z:=A+z(\tilde{A}-A)/\Norm{\tilde{A}-A}{\infty}$. Then $A_0=A$, $A_{z_1}=\tilde{A}$ for $z_1=\Norm{\tilde{A}-A}{\infty}$, and $z\mapsto A_z$ is holomorphic. For $z\in\disc(0,\delta)$, where $\delta$ is small enough, $DA_z$ has a bounded $H^{\infty}$ functional calculus in $L^p$ by Corollary~\ref{cor:DB}, and $z\mapsto f(DA_z)$ is holomorphic for $f\in H^{\infty}(S_{\theta})$ by Proposition~\ref{prop:holo}. By the Schwarz Lemma,
\begin{equation*}
  \Norm{f(DA_0)u-f(DA_{z_1})u}{p} \lesssim \abs{z_1}\,\Norm{f}{\infty}\Norm{u}{p},
\end{equation*} 
which gives the assertion. 
\end{proof}

We finally turn to the Lipschitz estimates for Hodge-Dirac operators with variable coefficients, using the same approach as in \cite[Section 10.1]{aamescorial}.

\begin{corollary} 
Let $1\leq p_1<p_2\leq\infty$, and let $\Pi_A$ and $\underline{\Pi}_A$ be Hodge-Dirac operators with variable coefficients, where $\Pi_A$ is R-bisectorial in $L^p$ and Hodge-decomposes $L^p$ for all $p\in(p_{1},p_{2})$. Then, for each $p\in(p_1,p_2)$, there exists $\delta=\delta_p>0$ such that, if $\Pi_B$ and $\underline{\Pi}_B$ are also Hodge-Dirac operators with variable coefficients with $\Norm{A_1-B_1}{\infty}+\Norm{A_2-B_2}{\infty}<\delta$,
then both $\Pi_{A}$ and $\Pi_{B}$ have a bounded $H^\infty$ functional calculus with some angle $\omega \in (0,\frac{\pi}{2})$, and for $\theta \in (\omega,\frac{\pi}{2})$, $f\in H^{\infty}(S_{\theta})$, and $u\in L^p$ there holds
\begin{equation*}
  \Norm{f(\Pi_A)u-f(\Pi_B)u}{p}
  \lesssim \max_{i=1,2}\Norm{A_i-B_i}{\infty}\Norm{f}{\infty}\Norm{u}{p}.
\end{equation*} 
\end{corollary}

\begin{proof} 
The philosophy of the proof is analogous to that of Corollary~\ref{cor:DB} but goes in the opposite direction: we now deduce results for operators of the form $\Pi_A$ from what we already know for the operators $DA$. To this end, consider the space $L^p\oplus L^p\oplus L^p$ and the operators
\begin{equation*}
  D:=\begin{pmatrix} 0 & 0 & 0 \\ 0 & 0 & \rGamma \\ 0 & \Gamma & 0 \end{pmatrix},\qquad
  A:=\begin{pmatrix} 0 & 0 & 0 \\ 0 & A_1 & 0 \\ 0 & 0 & A_2 \end{pmatrix}.
\end{equation*}

Let us write $\proj^A_0$, $\proj^A_{\Gamma}$ and $\proj^A_{\rGammaSmall_A}$ for the Hodge-projections associated to $\Pi_A$. By \eqref{coercond}, the restriction $A_1:\overline{\range_p(\rGamma)}\to\overline{\range_p(\rGamma_A)}$ is an isomorphism, and we write $A_1^{-1}$ for its inverse. Then a computation shows that
\begin{equation*}
  (I+itDA)^{-1}
  =\begin{pmatrix}
     I & 0 & 0 \\
     0 & A_1^{-1}\proj_{\rGammaSmall_A}^A(I+t^2\Pi_A^2)^{-1}A_1 & -it\rGamma A_2(I+t^2\Pi_A^2)^{-1} \\
     0 & -it\Gamma(I+t^2\Pi_A^2)^{-1}A_1 & \proj_{\Gamma}^A(I+t^2\Pi_A^2)^{-1} \end{pmatrix},
\end{equation*}
and one can check that the R-bisectoriality and the Hodge-decomposition of $\Pi_A$ imply the R-bisectoriality of $DA$. Indeed, for the diagonal elements above it is immediate, and for the non-diagonal elements it follows after writing $\rGamma A_2=A_1^{-1}\rGamma_A=A_1^{-1}\proj_{\rGammaSmall_A}^A\Pi_A$ and $\Gamma=\proj_{\Gamma}^A\Pi_A$.

We next define operators $S_A:L^p\to L^p\oplus L^p\oplus L^p$ and $T_A:L^p\oplus L^p\oplus L^p\to L^p$ by
\begin{equation*}
\begin{split}
  S_A u &:=(\proj^A_0 u,A_1^{-1}\proj^A_{\rGammaSmall_A}u,\proj^A_{\Gamma}u), \\
  T_A (u,v,w) &:=\proj^A_0 u+\proj^A_{\Gamma}w+\proj^A_{\rGammaSmall_A}A_1 v.
\end{split}
\end{equation*}
One checks that $T_A S_A=I$ and $S_A\Pi_A=(DA)S_A$. Hence $S_A(\lambda-\Pi_A)^{-1}=(\lambda-DA)^{-1}S_A$, and then by the definition of the functional calculus,
\begin{equation*}
   f(\Pi_A)u=T_A f(DA)S_A u,\qquad f\in H^{\infty}(S_{\theta}),\quad u\in L^p.
\end{equation*}

We repeat the above definitions and observations with $B$ in place of $A$, and then
\begin{equation*}
\begin{split}
  f(\Pi_A)u-f(\Pi_B)u
  =[T_A-T_B]f(DA)S_A u + T_B[f(DA)-f(DB)]S_A u + T_B f(DB)[S_A-S_B] u.
\end{split}
\end{equation*}
The asserted Lipschitz estimate then follows by using Corollary~\ref{cor:LipDA} for the middle term, and Proposition~\ref{prop:Hodgepert} and Corollary~\ref{cor:Hodgepert} for the other two terms. 
\end{proof}

\end{document}